\documentclass[11pt]{amsart}
\usepackage{amsmath,amssymb,amscd,enumerate}
\usepackage[all]{xy}
\usepackage[pdftex]{graphicx,color}
\usepackage{url}

\usepackage[colorlinks,linkcolor=blue,citecolor=blue,pdfstartview=FitH]{hyperref}
\usepackage{tikz}
\usepackage[top=1.1in, bottom=1.1in, left=1.1in, right=1.1in]{geometry}

\usetikzlibrary{arrows,shapes,positioning}
\usetikzlibrary{decorations.markings}
\tikzstyle arrowstyle=[scale=1]
\tikzstyle directed=[postaction={decorate,decoration={markings,
    mark=at position .55 with {\arrow[arrowstyle]{stealth}}}}]
\tikzstyle ddirected=[postaction={decorate,decoration={markings,
    mark=at position .45 with {\arrow[arrowstyle]{stealth}},
    mark=at position .55 with {\arrow[arrowstyle]{stealth}}}}]
\tikzstyle reverse directed=[postaction={decorate,decoration={markings,
    mark=at position .55 with {\arrowreversed[arrowstyle]{stealth};}}}]
\tikzstyle reverse ddirected=[postaction={decorate,decoration={markings,
    mark=at position .55 with {\arrowreversed[arrowstyle]{stealth};},
    mark=at position .65 with {\arrowreversed[arrowstyle]{stealth};}}}]

\newcommand{\co}{\colon \thinspace}

\newcommand\frakc{\mathfrak{c}}
\newcommand\calp{\mathcal{P}}
\newcommand\calq{\mathcal{Q}}

\newcommand\cals{\mathcal{S}}

\newcommand\maxi{\mathit{max}}

\theoremstyle{plain}
\newtheorem{para}{}[section]
\newtheorem{theorem}[para]{Theorem}
\newtheorem{proposition}[para]{Proposition}
\newtheorem{lemma}[para]{Lemma}

\theoremstyle{definition}
\newtheorem{question}[para]{Question}

\newtheorem{definition}[para]{Definition}
\newtheorem{example}[para]{Example}
\newtheorem*{notation}{Notation}
\newtheorem{stuff}[para]{}

\theoremstyle{remark}
\newtheorem{claim}{Claim}

\begin{document}

\title{Bounds for several-disk packings of hyperbolic surfaces}
\author{Jason DeBlois}
\address{Department of Mathematics, University of Pittsburgh}\email{jdeblois@pitt.edu}

\begin{abstract}  For any given $k\in\mathbb{N}$, this paper gives upper bounds on the radius of a packing of a complete hyperbolic surface of finite area by $k$ equal-radius disks in terms of the surface's topology.  We show that these bounds are sharp in some cases and not sharp in others.\end{abstract}

\maketitle

By a \textit{packing} of a metric space we will mean a collection of disjoint open metric balls.  This paper considers packings of a fixed radius on finite-area hyperbolic (i.e. constant curvature $-1$) surfaces, and our methods are primarily those of low-dimensional topology and hyperbolic geometry.  But before describing the main results in detail I would like to situate them in the context of the broader question below, so I will first list some of its other instances, survey what is known toward their answers, and draw analogies with the setting of this paper:

\begin{question}\label{broad}  For a fixed $k\in\mathbb{N}$ and topological manifold $M$ that admits a complete constant-curvature metric of finite volume, what is the supremal density of packings of $M$ by $k$ balls of equal radius, taken over all such metrics with fixed curvature? (Here the \textit{density} of a packing is the ratio of the sum of the balls' volumes to that of $M$.)\end{question}

The botanist P.L.~Tammes posed a positive-curvature case of Question \ref{broad} which is now known as \textit{Tammes' problem}, where $M = \mathbb{S}^2$ with its (rigid) round metric, in 1930.  Answers (i.e.~sharp density bounds) are currently known for $k\leq 14$ and $k=24$ after work of many authors, with the $k=14$ case only appearing in 2015 \cite{MusTar} (the problem's history is surveyed in \S 1.2 there).  

In the Euclidean setting, the case of Question \ref{broad} with $M$ an $n$-dimensional torus is equivalent to the famous \textit{lattice sphere packing problem}, see eg. \cite{ConSl}, when $k=1$.  When $k>1$ it is the \textit{periodic packing problem}; and finding the supremum over all $k$ is equivalent to the \textit{sphere packing problem}, which is solved only in dimensions $2$, $3$ \cite{Hales} and, very recently, $8$ \cite{Viazovska} and $24$ \cite{Via_etal}.  The lattice sphere packing problem is solved in all additional dimensions up to $8$, see \cite[Table 1.1]{ConSl}.

The hyperbolic case is also well studied.  Here a key tool, known in low-dimensional topology as ``B\"or\"oczky's theorem", asserts that any packing of $\mathbb{H}^n$ by balls of radius $r$ has \textit{local} density bounded above by the density in an equilateral $n$-simplex with sidelength $2r$ of its intersection with balls centered at its vertices \cite{Bor_bound}.  Rogers proved the analogous result for Euclidean packings earlier \cite{Rogers}.  We note that in the hyperbolic setting the bound depends on $r$ as well as $n$.

B\"or\"oczky's theorem yields bounds towards an answer to Question \ref{broad} for an arbitrary $k\in\mathbb{N}$ and complete hyperbolic manifold $M$, since a packing of $M$ has a packing of $\mathbb{H}^n$ as its preimage under the universal cover $\mathbb{H}^n\to M$.  Analogously, Rogers' result yields bounds toward the lattice sphere packing problem, which until recently were still the best known in some dimensions (see \cite[Appendix A]{CohnZhao}).  This basic observation is particularly useful in low dimensions: for instance Rogers' lattice sphere packing bound is sharp only in dimension $2$ (where it is usually attributed to Gauss).  In the three-dimensional hyperbolic setting, where B\"or\"oczky's Theorem was actually proved earlier by B\"or\"oczky--Florian \cite{BoFlo}, its ``$r=\infty$'' (i.e.~horoball packing) case yields sharp lower bounds on the volumes of cusped hyperbolic 3-orbifolds \cite{Meyerhoff} and 3-manifolds \cite{Adams}, for example.

In dimension two, C.~Bavard observed that B\"or\"oczky's theorem implies answers to the $k=1$ case of Question \ref{broad} for every closed hyperbolic surface $F$ \cite{Bavard}; that is, it yields sharp bounds on the maximal radius of a ball embedded in such a surface.  The same argument yields bounds towards Question \ref{broad} for arbitrary $k$ and hyperbolic surfaces $F$, as we will observe in Section \ref{sharpness}.  This was already shown for closed genus-two surfaces when $k=2$, by Kojima--Miyamoto \cite{KM}.

For a non-compact surface $F$ of finite area, it is no longer true that a maximal-density packing of $F$ pulls back to a packing of $\mathbb{H}^2$ with maximal local density: the cusps of $F$ yield ``empty horocycles", large vacant regions in the preimage packing.  In \cite{DeB_Voronoi} I introduced a new technique for proving two-dimensional packing theorems and used it to settle the $k=1$ case of Question \ref{broad} for arbitrary complete, orientable hyperbolic surfaces of finite area.  M.~Gendulphe has since released a preprint which resolves the non-orientable $k=1$ cases by another method \cite{Gendulphe}.  But the restriction to orientable surfaces in \cite{DeB_Voronoi} is not necessary, as I show in Section \ref{bounds} below.  The first main result here records density bounds for arbitrary $k\in\mathbb{N}$ and complete, finite-area hyperbolic surfaces (orientable or not) which follow from the packing theorems of \cite{DeB_Voronoi}.

In fact we bound the \textit{radius} of packings.  But this is equivalent to bounding their density, since the area of a hyperbolic disk is determined by its radius, and by the Gauss--Bonnet theorem the area of a complete, finite-area hyperbolic surface is determined by its topological type.

\newcommand\VorBound{For $\chi< 0$ and $n\geq 0$, and each $k\in\mathbb{N}$, the radius $r$ of a packing of a complete, finite-area, $n$-cusped hyperbolic surface $F$ with Euler characteristic $\chi$ by $k$ disks of equal radius satisfies $r\leq r_{\chi,n}^k$, where $r_{\chi,n}^k$ is the unique solution to:\begin{align*}
  \left(6-\frac{6\chi+3n}{k}\right)\alpha(r_{\chi,n}^k) + \frac{2n}{k}\beta(r_{\chi,n}^k) & = 2\pi, \end{align*}
Here $\alpha(r)= 2\sin^{-1}\left(\frac{1}{2\cosh r}\right)$ and $\beta(r)= \sin^{-1}\left(\frac{1}{\cosh r}\right)$ each measure vertex angles, respectively those of an equilateral triangle with sides of length $2r$, and at finite vertices of a horocyclic ideal triangle with compact side length $2r$.  

A complete, finite-area, $n$-cusped hyperbolic surface $F$ has an equal-radius packing by $k$ disks of radius $r=r_{\chi,n}^k$ if and only if $F$ decomposes into equilateral and horocyclic ideal triangles, all with compact sidelength $2r_{\chi,n}^k$ and intersecting pairwise (if at all) only at vertices or along entire edges, with $k$ vertices and exactly $n$ horocyclic ideal triangles.

In this case, the disk centers $x_1,\hdots,x_k$ are the vertices of such a decomposition of $F$, and conversely, such a decomposition is canonically obtained from the Delaunay tessellation of $F$ determined by $\{x_1,\hdots,x_k\}$ (in the sense of \cite[Cor.~6.27]{DeB_Delaunay}) by subdividing each horocyclic two cell with a single ray from its vertex that exits its cusp.}
\theoremstyle{plain}
\newtheorem*{VorBoundProp}{Proposition \ref{Vor bound}}
\begin{VorBoundProp}\VorBound\end{VorBoundProp}

Above, a \textit{horocyclic ideal triangle} is the convex hull in $\mathbb{H}^2$ of three points, two of which lie on a horocycle $C$ with the third at the ideal point of $C$.  We prove Proposition \ref{Vor bound} in Section \ref{bounds}.

To my knowledge, the bounds of Proposition \ref{Vor bound} are the best in the literature for every $k$, $\chi$ and $n$.  They coincide with those from B\"or\"oczky's Theorem in the closed ($n=0$) case but are otherwise stronger, see Proposition \ref{sharper}.  But as we show below, they are not sharp in general; indeed, it is easy to see that they are not \textit{attained} in general.  For a closed surface $F$ with Euler characteristic $\chi$ attaining the bound of $r_{\chi,0}^k$, the equilateral triangles that decompose $F$ all have equal vertex angles, since they have equal side lengths.  It follows that its triangulation must be \textit{regular}; that is, each vertex must see the same number of triangles.  This imposes the condition that $k$ divide $6\chi$, since the number of such triangles is $2(k-\chi)$ by a simple computation.

In this sense Proposition \ref{Vor bound} is akin to L.~Fejes T\'oth's general bound toward Tammes' problem \cite{Toth}, which is attained only for $k=3$, $4$, $6$, and $12$: those for which $S^2$ has a regular triangulation with $k$ vertices. (The last three are realized by the boundaries of a tetrahedron, octahedron, and icosahedron, respectively).  When the bound is attained in the non-compact $(n>0)$ case, one correspondingly expects both equilateral and horocyclic ideal triangle vertices to be ``evenly distributed''.  Lemma \ref{i_n_j} formulates this precisely, and shows that it imposes the conditions that $k$ divide both $6\chi$ and $n$.  The second main result of this paper shows that the bound is attained when these conditions hold.

\newcommand\Attained{For $\chi< 0$ and $n\geq0$, if $k\in\mathbb{N}$ divides both $n$ and $6\chi$ then there exists a complete hyperbolic surface $F$ of finite area with Euler characteristic $\chi$ and $n$ cusps and a packing of $F$ by $k$ disks of radius $r_{\chi,n}^k$.  If $\chi+n$ is even then $F$ may be taken orientable or non-orientable.}
\newtheorem*{attained theorem}{Theorem \ref{Vor bound attained}}
\begin{attained theorem}\Attained\end{attained theorem}

The non-trivial part of the proof is the purely topological Proposition \ref{all triangulations}, which constructs closed surfaces of Euler characteristic $\chi+n$, triangulated with $k+n$ vertices of which $n$ have valence one, when $k$ divides both $6\chi$ and $n$.  (The valence-one vertices correspond to cusps of $F$.)  Previous work of Edmonds--Ewing--Kulkarni \cite{EdEwKu} covers the $n=0$ cases of Proposition \ref{all triangulations} but its techniques, which exploit the existence of certain branched coverings in this setting, do not naturally extend to the $n>0$ cases.  The proof here is independent of \cite{EdEwKu}, even for $n=0$.

I could not prove in the $n>0$ case that the bound $r_{\chi,n}^k$ is attained \textit{only} if $k$ divides both $6\chi$ and $n$.  But I do show in Lemma \ref{this ain't it} that for a given $\chi<0$ and $n>0$, it is attained only at finitely many $k$.  Section \ref{dullity} goes on to establish:

\newcommand\NotAttained{For any $\chi < 0$ and $n\geq 0$, and $k\in\mathbb{N}$, the bound $r_{\chi,n}^k$ is sharp if and only if it is attained.  For $n=0$, $r_{\chi,0}^k$ is attained if and only if $k$ divides $6\chi$.  For any fixed $\chi<0$ and $n>0$, there are only finitely many $k$ for which $r_{\chi,n}^k$ is attained.}
\newtheorem*{not attained theorem}{Theorem \ref{Vor bound not attained}}
\begin{not attained theorem}\NotAttained\end{not attained theorem}

This follows immediately from Lemma \ref{this ain't it} and the main result Proposition \ref{there's a max} of Section \ref{dullity}, which asserts for each $\chi$ and $n$ that the function on the entire moduli space of all (orientable or non-orientable) hyperbolic surfaces with Euler characteristic $\chi$ and $n$ cusps which measures the maximal $k$-disk packing radius does attain a maximum.  (This is not obvious because the function in question is not proper, see \cite[Cor.~0.1]{DeB_Voronoi} in the case $k=1$.) Key to the proof of Proposition \ref{there's a max} is Lemma \ref{the FLoBUS}, which asserts that the thick part of a surface with a short geodesic can be inflated while increasing the length of the geodesic.  Proposition \ref{there's a max} and its proof were suggested by a referee for \cite{DeB_locmax}, and sketched by Gendulphe \cite{Gendulphe}, in the case $k=1$.  

I'll end this introduction with a couple of further questions.  First, recall from the discussion above that Lemma \ref{this ain't it} is is likely not best possible.  It would be interesting to know whether the bound of Theorem \ref{Vor bound} is attained in \textit{any} cases beyond those covered by Theorem \ref{Vor bound attained}.

\begin{question}  For $\chi< 0$ and $n> 0$, does there exist any $k$ not dividing both $n$ and $6\chi$ for which a complete hyperbolic surface of finite area with $n$ cusps and Euler characteristic $\chi$ admits a packing by $k$ disks that each have radius $r_{\chi,n}^k$?\end{question}

For any example not covered by Theorem \ref{Vor bound attained}, some observations from the proof of Lemma \ref{this ain't it} can be used to show that $\alpha(r_{\chi,n}^k)$ and $\beta(r_{\chi,n}^k)$ must satisfy an additional algebraic dependence beyond what is prescribed in Theorem \ref{Vor bound}.  While this seems unlikely, it is not clear (to me at least) that it does not ever occur.  In a follow-up paper I will consider the case of two disks on the three-punctured sphere, showing that at least in this simplest possible case it does not.

It would also be interesting to know whether $k$-disk packing radius has local but non-global maximizers.  We phrase the question below in the language introduced above Proposition \ref{there's a max}.

\begin{question}  For which smooth surfaces $\Sigma$ and $k\in\mathbb{N}$ does the function $\maxi_k$ from Definition \ref{maximy max max} have a local maximum on $\mathfrak{T}(\Sigma)$ that is not a global maximum?  In particular, can this occur if $k$ divides both $6\chi$ and $n$, where $\chi$ is the Euler characteristic of $\Sigma$ and $n$ its number of cusps?\end{question}

Gendulphe answered this ``no'' for $k=1$ \cite{Gendulphe}, extending my result on the orientable case \cite{DeB_locmax}.

\subsection*{Acknowledgements}  Thanks to Dave Futer for a keen observation, and for pointing me to \cite{EdEwKu}, and to Ian Biringer for a helpful conversation.  Thanks also to the anonymous referee for a careful reading and helpful comments.

\section{Decompositions of orientable and non-orientable surfaces}\label{bounds}

The bound of \cite[Theorem 5.11]{DeB_Voronoi}, which Proposition \ref{Vor bound} generalizes, is proved using the \textit{centered dual complex plus} as defined in Proposition 5.9 of \cite{DeB_Voronoi}.  This is a decomposition of a finite-area hyperbolic surface $F$ that is canonically determined by a finite subset of $F$.  In this section we will first recap the construction of the centered dual plus, and show that \cite[Prop.~5.9]{DeB_Voronoi} carries through to the non-orientable case without revision.  Then we will prove Proposition \ref{Vor bound}.

Let $F$ be a complete, finite-area hyperbolic surface, $\cals\subset F$ a finite set, and $\pi\co\mathbb{H}^2\to F$ a locally isometric universal cover.  We will assume here that $F$ is non-orientable, since the orientable case is covered by previous work, and let $F_0$ be the orientable double cover of $F$ and $\cals_0$ be the preimage of $\cals$ in $F_0$.  Note that $\pi$ factors through a locally isometric universal cover $p\co\mathbb{H}^2\to F_0$, so in fact $\widetilde{\cals}=p^{-1}(\cals_0)$.
This set is invariant under the isometric actions of $\pi_1 F$ and $\pi_1 F_0$ by covering transformations.

Theorem 5.1 of \cite{DeB_Voronoi}, which is a rephrasing of \cite[Thrm.~6.23]{DeB_Delaunay} for surfaces, asserts the existence of a $\pi_1 F_0$-invariant \textit{Delaunay tessellation} of $p^{-1}(\cals_0)$ characterized by the following \textit{empty circumcircles condition}:

\begin{quote}\it  For each circle or horocycle $S$ of $\mathbb{H}^2$ that intersects $\widetilde{\cals}$ and bounds a disk or horoball $B$ with $B\cap\widetilde{\cals} = S\cap\widetilde{\cals}$, the closed convex hull of $S\cap\widetilde\cals$ in $\mathbb{H}^2$ is a Delaunay cell. Each Delaunay cell has this form.\end{quote}

Since this characterization is in purely geometric terms it implies that the Delaunay tessellation is invariant under every isometry that leaves $\widetilde{\cals}$ invariant.  In particular, it is $\pi_1 F$-invariant.  In fact the entire \cite[Th.~5.1]{DeB_Voronoi} extends to the non-orientable case; one needs only additionally observe that the parabolic fixed point sets of $\pi_1 F$ and $\pi_1 F_0$ are identical. We will now run through the remaining results of Sections 5.1 and 5.2 of \cite{DeB_Voronoi}, which build to Proposition 5.9 there, and comment on their extensions to the non-orientable case.  

Corollary 5.2 of \cite{DeB_Voronoi} makes three assertions about properties of the Delaunay tessellation's image in the quotient surface, which all extend directly to the non-orientable case.  The first, on finiteness of the number of $\pi_1 F_0$-orbits of Delaunay cells, implies the same for $\pi_1 F$-orbits.  The original proof (in \cite[Cor.~6.26]{DeB_Delaunay}) of the the second, on interiors of compact cells embedding, applies directly.  In our setting, the third assertion is that for each non-compact Delaunay cell $C_u$, which is invariant under some parabolic subgroup $\Gamma_u$ of $\pi_1 F_0$, $p|_{\mathit{int}\,C_u}$ factors through an embedding of $\mathit{int}\,C_u/\Gamma_u$ to a cusp of $F_0$.  Here we have:

\begin{lemma}  The stabilizer of $C_u$ in $\pi_1 F$ is also $\Gamma_u$; each cusp of $F_0$ projects homeomorphically to $F$; and $\pi|_{\mathit{int}\,C_u}$ factors through an embedding of $\mathit{int}\,C_u/\Gamma_u$ to a cusp of $F$.\end{lemma}

\begin{proof}  The stabilizer of $C_u$ in the full group of isometries of $\mathbb{H}^2$ is the stabilizer of the ideal point $u$ of the horocycle $S$ in which it is inscribed.  Using the upper half-plane model and translating $u$ to $\infty$, standard results on the classification of isometries (see eg.~\cite[Th.~1.3.1]{Katok}) imply that this group is the semidirect product
$$ \left\{\left(\begin{smallmatrix} 1 & x \\ 0 & 1\end{smallmatrix}\right)\,:\, x\in\mathbb{R} \right\} \rtimes \{z\mapsto -\bar{z}\} $$
with the index-two translation subgroup preserving orientation.  One sees directly from the classification that every orientation-reversing element here reflects about a geodesic.  Since $\pi_1 F$ acts freely on $\mathbb{H}^2$, every element that stabilizes $C_u$ preserves orientation.  But $\pi_1 F_0$ consists precisely of those elements of $\pi_1 F$ that preserve orientation.  The lemma's first assertion follows directly, and the latter two follow from that one.\end{proof}

The main result of Section 5.1 of \cite{DeB_Voronoi}, Corollary 5.6, still holds if its orientability hypothesis is dropped.  This implies for us that the \textit{centered dual complex} of $\widetilde{S}$ is $\pi_1 F$-invariant and its two-cells project to $F$ homeomorphically on their interiors.  The centered dual complex is constructed in \cite[\S 2]{DeB_Voronoi}; in particular see Definition 2.26 there, and it is the object to which the main packing results Theorems 3.31 and 4.16 of \cite{DeB_Voronoi} apply.

The proof of \cite[Cor.~5.6]{DeB_Voronoi} given there in fact applies without revision in the non-orientable case.  To this end we note that the \textit{Voronoi tessellation} of $\widetilde{\cals}$ and its \textit{geometric dual} (see \cite[Th.~1.2]{DeB_Voronoi} and the exposition above it) are by their construction invariant under every isometry preserving $\widetilde{\cals}$.  And the proofs of Lemmas 5.4 and 5.5, and Corollary 5.6, do not use the hypothesis that elements of $\pi_1 F$ preserve orientation, only that they act isometrically and are fixed point-free. 

Section 5.2 of \cite{DeB_Voronoi} builds to the description of the centered dual complex plus in Proposition 5.9, its final result.  The issue here is that if $F$ has cusps then the underlying space of the centered dual decomposition is not necessarily all of $\mathbb{H}^2$, but rather the union of all geometric dual cells (which are precisely the compact cells of the Delaunay tessellation, by \cite[Remark 5.3]{DeB_Voronoi}) with possibly some horocyclic ideal triangles.  But Lemma 5.8 of \cite{DeB_Voronoi} shows that each non-compact Delaunay cell $C_u$ intersects this underlying space in a sub-union of the collection of horocyclic ideal triangles obtained by joining each of its vertices to its ideal point $u$ by a geodesic ray.  So the centered dual plus is obtained by simply adding two-cells (and their edges and ideal vertices) to the centered dual, one for each horocyclic ideal triangle  obtained from each $C_u$ as above that does not already lie in a centered dual two-cell.

The proof of \cite[Lemma 5.8]{DeB_Voronoi} again extends without revision to the non-orientable  setting, and provides the final necessary ingredient for:

\begin{proposition}\label{fivepointnine}  Proposition 5.9 of \cite{DeB_Voronoi} still holds if $F$ is not assumed oriented.\end{proposition}

Having established this, we turn to the first main result.

\begin{proposition}\label{Vor bound}\VorBound\end{proposition}

\begin{proof} The first part of this proof closely tracks that of Theorem 5.11 in \cite{DeB_Voronoi}.  Given a complete, finite-area hyperbolic surface $F$ of Euler characteristic $\chi$ with $n$ cusps, equipped with an equal-radius packing by $k$ disks of radius $r$, we let $\cals$ be the set of disk centers.  Fixing a locally isometric universal cover $\pi\co \mathbb{H}^2\to F$, let $\widetilde{\cals} = \pi^{-1}(\cals)$ and, applying Proposition \ref{fivepointnine}, enumerate a complete set of representatives for the $\pi_1 F$-orbits of cells of the centered dual complex plus as $\{C_1,\hdots,C_m\}$.  By the Gauss--Bonnet theorem the $C_i$ satisfy:
$$ \mathrm{Area}(C_1) + \hdots + \mathrm{Area}(C_m) = -2\pi\chi $$
Take $C_i$ non-compact if and only if $i\leq m_0$ for some fixed $m_0\leq m$, and for each $i\leq m$ let $n_i$ be the number of edges of $C_i$.  Each compact edge of the decomposition has length at least $d = 2r$, since it has disjoint open subsegments of length $2r$ around each of its vertices.  For $i\leq m_0$ we therefore have the key area inequality obtained from \cite[Th.~4.16]{DeB_Voronoi} and some calculus:
$$\mathrm{Area}(C_i) \geq D_0(\infty,d,\infty) + (n_i-3)D_0(d,d,d),$$
with equality holding if and only if $n_i=3$ and the compact side length is $d$.  Here $D_0(\infty,d,\infty)$ is the area of a horocyclic ideal triangle with compact side length $d$, and $D_0(d,d,d)$ is the area of an equilateral hyperbolic triangle with all side lengths $d$.

For $m_0 < i \leq m$, again as described in the proof of Theorem 5.11 of \cite{DeB_Voronoi}, Theorem 3.31 there and some calculus give
$$ \mathrm{Area}(C_i) \geq (n_i-2) D_0(d,d,d), $$
with equality if and only if $n_i=3$ and all side lengths are $d$.  We therefore have:\begin{align}
  -2\pi\chi & \geq m_0 \cdot D_0(\infty,d,\infty) + \left(\sum_{i=1}^m (n_i-2) - m_0\right)\cdot D_0(d,d,d) \nonumber\\
 \label{middle} & \geq n \cdot D_0(\infty,d,\infty) + \left(\sum_{i=1}^m n_i - 2m - n\right) \cdot D_0(d,d,d) \\
    & = n\cdot (\pi-2\beta(r)) + (2e-2m - n)\cdot (\pi-3\alpha(r))\nonumber\end{align}
In moving from the first to the second line above, we have used the fact that $m_0\geq n$ and $D_0(\infty,d,\infty) > D_0(d,d,d)$ (again see the proof of \cite[Th.~5.11]{DeB_Voronoi}) to trade $m_0-n$ instances of $D_0(\infty,d,\infty)$ down for the same number of $D_0(d,d,d)$.  In moving from the second to the third we have rewritten $\sum_{i=1}^m n_i$ as $2e$, where $e$ is the total number of edges, and used the angle deficit formula for areas of hyperbolic triangles (recall from above that $d=2r$).

We will now apply an Euler characteristic identity satisfied by the closed surface $\bar{F}$ obtained from $F$ by compactifying each cusp with a unique marked point.  By Proposition \ref{fivepointnine} the $C_i$ project to faces of a cell decomposition of $\bar{F}$.  This satisfies $v-e+f = \chi(\bar{F})$, where $v$, $e$ and $f$ are the total number of vertices, edges, and faces, respectively.  For us this translates to
$$ n+k - e + m = \chi+n, $$
since $\chi(\bar{F}) = \chi+n$.  Here there are $n+k$ vertices of the centered dual plus decomposition of $\bar{F}$: $k$ that are disk centers, and $n$ that are marked points.  As above we have called the number of edges $e$, and the number of faces is $m$.  Now substituting $2k-2\chi$ for $2e-2m$ on the final line of (\ref{middle}), after simplifying we obtain
$$  \left(6 - \frac{6\chi+3n}{k}\right)\alpha(r) + \frac{2n}{k}\beta(r) \geq 2\pi  $$
Since $\alpha$ and $\beta$ are decreasing functions of $r$ we obtain that $r\leq r_{\chi,n}^k$.  

As we noted above the inequalities (\ref{middle}), equality holds on the first line if and only if $n_i=3$ for each $i$, i.e.~each $C_i$ is a triangle, and all compact sides have length $d=2r$.  This implies in particular that the compact centered dual two-cells (the $C_i$ for $i>m_0$) are all Delaunay two-cells, and the non-compact cells are obtained by dividing horocyclic Delaunay cells of $\widetilde\cals$ into horocyclic ideal triangles by rays from vertices.  As we noted below (\ref{middle}), equality holds on the second line if and only if $m_0 = n$; that is, each cusp corresponds to a unique horocyclic ideal triangle.  Thus when $r = r^{k}_{\chi,n}$, $F$ decomposes into equilateral and $n$ horocyclic ideal triangles with all compact sidelengths equal to $2r_{\chi,n}^k$.

In this case, one finds by inspecting the definition of the centered dual plus that every compact centered dual two-cell is a Delaunay triangle, and every non-compact cell is obtained from a Delaunay monogon by subdividing it by a single arc from its vertex out its cusp.  To prove the Proposition it thus only remains to show that a surface  of Euler characteristic $\chi$ and $n$ cusps which decomposes into equilateral triangles and $n$ horocyclic ideal triangles, all with compact sidelength $2r_{\chi,n}^k$, with $k$ vertices, admits a packing by $k$ disks of radius $r_{\chi,n}^k$ centered at the vertices.

This follows the line of argument from Examples 5.13 and 5.14 of \cite{DeB_Voronoi}.  Lemma 5.12 there implies in particular that an open metric disk of radius $r_{\chi,n}^k$ centered at a vertex $v$ of an equilateral triangle $T$ in $\mathbb{H}^2$ with all sidelengths $2r_{\chi,n}^k$ intersects $T$ in a full sector with angle measure equal to the interior angle of $T$ at $v$.  It is not hard to show that the same holds for a disk centered at a finite vertex of a horocyclic ideal triangle.  In both cases it is clear moreover that disks of radius $r$ centered at distinct (finite) vertices of such triangles do not intersect.  Therefore for a surface decomposed into a collection of equilateral and horocyclic ideal triangles with all compact sidelengths $2r_{\chi,n}^k$, a collection of disjoint embedded open disks of radius $r_{\chi,n}^k$ centered at the vertices of the decomposition is assembled from disk sectors in each triangle around each of its vertices.
\end{proof}

\section{Comparing bounds}\label{sharpness}

In this section we give brief accounts of two arguments that give alternative bounds on the $k$-disk packing radius: a naive bound $r_A^k$ and a bound from B\"or\"oczky's Theorem \cite{Bor_bound} that turns out to simply be $r_{\chi,0}^k$.  These arguments are not original, in particular the latter can essentially be found in \cite{Bavard} and \cite{KM}.  We will then compare the bounds in Proposition \ref{sharper}.

\begin{stuff}[The naive bound]  If $k$ disks of radius $r$ are packed in a complete finite-area surface $F$ then the sum of their areas is no more than that of $F$.  The area of a hyperbolic disk of radius $r$ is $2\pi(\cosh r - 1)$ (see eg.~\cite[Exercise 3.4(1)]{Ratcliffe}), and the Gauss--Bonnet theorem implies that the area of a complete, finite-area hyperbolic surface $F$ with Euler characteristic $\chi$ is $-2\pi\chi$.  Therefore $2\pi(\cosh r -1)\cdot k \leq -2\pi\chi$, whence $r \leq r_A^k$ defined by
\begin{align}\label{naive} \cosh \left(r_A^k\right) = 1-\frac{\chi}{k} \end{align}
\end{stuff}

\begin{stuff}[B\"or\"oczky's bound]  Suppose again that $k$ disks of radius $r$ are packed in a complete, finite-area surface $F$.  Fix a locally isometric universal cover $\pi\co\mathbb{H}^2\to F$ and consider the the preimage of the disks packed on $F$, which is a packing of $\mathbb{H}^2$ by disks of radius $r$.  Each disk $\widetilde{D}$ in the preimage determines a Voronoi $2$-cell $\widetilde{V}$ (see eg.~Section 1 of \cite{DeB_Voronoi}), and for $\alpha$ as in Theorem \ref{Vor bound} B\"or\"oczky's theorem asserts the following bound on the density of $\widetilde{D}$ in $\widetilde{V}$:
$$ \frac{\mathrm{Area}(\widetilde{D})}{\mathrm{Area}(\widetilde{V})} \leq \frac{3\alpha(r)(\cosh r -1)}{\pi-3\alpha(r)}  \quad\Rightarrow\quad  2\pi\left(\frac{\pi}{3\alpha(r)} - 1\right)\leq \mathrm{Area}(\widetilde{V})$$
We obtain the right-hand inequality above upon substituting for $\mathrm{Area}(\widetilde{D})$ and simplifying.

The packing of $\mathbb{H}^2$ by the preimage of the disks on $F$ is invariant under the action of $\pi_1 F$ by covering transformations, so this is also true of its Voronoi tessellation.  Moreover since there is a one-to-one correspondence between disks and Voronoi $2$-cells, a full set $\{\widetilde{D}_1,\hdots,\widetilde{D}_k\}$ of disk orbit representatives determines a full set $\{\widetilde{V}_1,\hdots,\widetilde{V}_k\}$ of Voronoi cell orbit representatives.  Their areas thus sum to that of $F$.  If $F$ has Euler characteristic $\chi$ then summing the right-hand inequalities above and applying the Gauss--Bonnet theorem yields
$$ 2\pi\left(\frac{\pi}{3\alpha(r)} - 1\right)\cdot k \leq -2\pi\chi \quad\Rightarrow\quad \pi \leq \left(1-\frac{\chi}{k}\right)3\alpha(r) $$
From the formula $\alpha(r) = 2\sin^{-1}(1/(2\cosh r))$, we see that $\alpha$ decreases with $r$.  Comparing with the equation defining $r_{\chi,n}^k$ in Theorem \ref{Vor bound}, we thus find that this inequality implies $r\leq r_{\chi,0}^k$.\end{stuff}

The relationship between the bounds $r_A^k$, $r_{\chi,0}^k$ and $r_{\chi,n}^k$ is perhaps not immediately clear from their formulas in all cases.  The result below clarifies this.

\begin{proposition}\label{sharper}  For any fixed $\chi<0$ and $n\geq 0$, and $k\in\mathbb{N}$, $r_{\chi,n}^k \leq r_{\chi,0}^k < r_A^k$.  The inequality $r_{\chi,n}^k\leq r_{\chi,0}^k$ is strict if and only if $n>0$.\end{proposition}

\begin{proof}  We first compare $r_{\chi,0}^k$ with $r_{\chi,n}^k$ when $n>0$.  We will apply Corollary 5.15 of \cite{DeB_Voronoi}, which asserts for any $r>0$ that $2\beta(r) < 3\alpha(r)$.  Slightly rewriting the equation defining $r_{\chi,n}^k$ gives:
\[ 2\pi = \left(6-\frac{6\chi}{k}\right)\alpha(r_{\chi,n}^k) + \frac{n}{k}\left(2\beta(r_{\chi,n}^k)-3\alpha(r_{\chi,n}^k)\right) < \left(6-\frac{6\chi}{k}\right)\alpha(r_{\chi,n}^k) \]
Since $\alpha$ decreases with $r$, and $r_{\chi,0}^k$ is defined by setting the right side of the inequality above equal to $2\pi$, it follows that $r_{\chi,n}^k < r_{\chi,0}^k$.

We now show that $r_{\chi,0}^k < r_A^k$ for each $\chi$ and $k$.  Recall that $\alpha(r) = 2\sin^{-1}\left(1/(2\cosh r)\right)$.  The inverse sine function is concave up on $(0,1)$ and takes the value $0$ at $0$ and $\pi/6$ at $1/2$, so $\alpha(r_{\chi,0}^k) < \pi/(3\cosh r_{\chi,0}^k)$.  Plugging back into the definition of $r_{\chi,0}^k$, and comparing with that of $r_A^k$, yields the desired inequality.\end{proof}

\section{Some examples showing sharpness}\label{examples}

\noindent In this section we'll prove Theorem \ref{Vor bound attained}, which asserts that the bound $r_{\chi,n}^k$ of Proposition \ref{Vor bound} is attained under certain divisibility hypotheses.  We begin with a simple counting lemma recording the  combinatorial condition that motivates these hypotheses.

\begin{lemma}\label{i_n_j}  Suppose $F$ is a complete, orientable hyperbolic surface of finite area with Euler characteristic $\chi<0$ and $n\geq 0$ cusps that decomposes into a collection of compact and horocyclic ideal triangles that intersect pairwise (if at all) only at vertices or along entire edges, such that there are $k$ vertices and exactly $n$ horocyclic ideal triangles.  

If there exist fixed $i$ and $j$ such that each vertex of the decomposition of $F$ is the meeting point of exactly $i$ compact and $j$ horocyclic ideal triangle vertices, then $k$ divides both $n$ and $6\chi$, and
\begin{align}\label{i and j} i = 6-\frac{6\chi+3n}{k} \quad\mbox{and}\quad j = \frac{2n}{k}. \end{align}\end{lemma}

\begin{proof}  The closed surface $\bar{F}$ obtained from $F$ by compactifying each cusp with a single point has Euler characteristic $\chi+n$, and the given decomposition determines a triangulation of $\bar{F}$ with $k+n$ vertices, where each horocyclic ideal triangle has been compactified by the addition of a single vertex at its ideal point.  Since $F$ has $n$ cusps and $n$ horocyclic ideal triangles, each horocyclic ideal triangle encloses a cusp, its non-compact edges are identified in $F$, and the quotient of these edges has one endpoint at the added vertex (which has valence one).

Noting that the numbers $e$, of edges, and $f$, of faces of the triangulation of $\bar{F}$ satisfy $2e = 3f$, computing its Euler characteristic gives:\begin{align}\label{Euler}
 v-e+f = k+n - f/2 = \chi+n \end{align}
Therefore $F$ has a total of $2(k-\chi)$ triangles, of which $2(k-\chi)-n$ are compact.  Since each compact triangle has three vertices and each horocyclic ideal triangle has two, even distribution of vertices determines the counts $i$ and $j$ above.  These imply in particular that $k$ must divide both $2n$ and $6\chi+3n$.  But we note that in fact $k$ must divide $n$, since the two vertices of each horocyclic ideal triangle are identified in $F$, and therefore $k$ must also divide $6\chi$.\end{proof}

A condition equivalent to the hypothesis of Lemma \ref{i_n_j} on a topological surface in fact ensures the existence of a hyperbolic structure satisfying the conclusion of Theorem \ref{Vor bound attained}.

\begin{proposition}\label{top to geom} For $\chi<0$ and $n\geq 0$, suppose $S$ is a closed surface with Euler characteristic $\chi+n$ and $n$ marked points that is triangulated with $k+n$ vertices, including the marked points, with the following properties:\begin{itemize}
  \item Each marked point is contained in exactly one triangle, which we also call \mbox{\rm marked}.
  \item There exist fixed $i,j\geq 0$ such that exactly $i$ non-marked and $j$ marked triangle vertices meet at each of the remaining $k$ vertices.\end{itemize}
Then there is a complete hyperbolic surface $F$ of finite area that decomposes into a collection of equilateral and exactly $n$ horocyclic ideal triangles intersecting pairwise only at vertices or along entire edges, if at all; and there is a homeomorphism $f\co S-\mathcal{P}\to F$, where $\mathcal{P}$ is the set of marked points, taking non-marked triangles to equilateral triangles and each marked triangle, less its marked vertex, to a horocyclic ideal triangle. 

$F$ is unique with this property up to isometry. That is, for any complete hyperbolic surface $F'$ and homeomorphism $f'\co S-\mathcal{P}\to F'$ satisfying the conclusion above, there is an isometry $\phi\co F'\to F$ such that $f$ is properly isotopic, preserving triangles, to $\phi\circ f'$. Also, $F$ has a packing by $k$ disks of radius $r_{\chi,n}^k$, each centered at the image of a non-marked vertex of $S$.\end{proposition}

Here a \textit{triangulation} of a surface, possibly with boundary, is simply a homeomorphism to the quotient space of a finite disjoint union of triangles by homeomorphically pairing certain edges.  If the surface has boundary then not all edges must be paired.

In this section we will also prove \textit{existence} of the triangulations required in Proposition \ref{top to geom}.

\begin{proposition}\label{all triangulations}  For any $\chi< 0$ and $n\geq 0$, and any $k\in\mathbb{N}$ that divides both $6\chi$ and $n$, there is a closed non-orientable surface with Euler characteristic $\chi+n$ and $n$ marked points that is triangulated with $k+n$ vertices, including the marked points, with the following properties:\begin{itemize}
  \item Each marked point is contained in exactly one triangle, which we also call \mbox{\rm marked}.
  \item Exactly $i$ non-marked and $j$ marked triangle vertices meet at each of the remaining $k$ vertices, where $i$ and $j$ are given by (\ref{i and j}).\end{itemize}
If $\chi+n$ is even then there is also an orientable surface of Euler characteristic $\chi+n$ with $n$ marked points and such a triangulation.\end{proposition}

The main result of this section follows directly from combining Proposition \ref{top to geom} with \ref{all triangulations}.

\begin{theorem}\label{Vor bound attained}\Attained\end{theorem}

\subsection{Geometric surfaces from triangulations}\label{geometric} We now proceed to prove the Propositions above. This subsection gives a standard argument to prove Proposition \ref{top to geom}.

\begin{proof} Let $S$ be a closed topological surface with a collection $\mathcal{P}$ of $n$ marked points, triangulated with $k+n$ vertices satisfying the Proposition's hypotheses. Applying the Euler characteristic argument from the proof of Lemma \ref{i_n_j} with $S$ in the role of $\bar{F}$ there gives that there are $2(k-\chi)-n$ non-marked triangles, and hence that $i$ and $j$ are given by (\ref{i and j}). By definition, $S$ is a quotient space of a disjoint union of triangles by pairing edges homeomorphically. We will produce our hyperbolic surface $F$ by taking a disjoint union of equilateral and horocyclic ideal triangles in $\mathbb{H}^2$ corresponding to the triangles of $S$ and pairing their edges to match.

Number the non-marked triangles of the disjoint union giving rise to $S$ from $1$ to $m$, where $m = 2(k- \chi)-n$.  Let $T_1,\hdots,T_m$ be a collection of disjoint equilateral triangles in $\mathbb{H}^2$ with side lengths $2r_{\chi,n}^k$, and for $1\leq s\leq m$ fix a homeomorphism from $T_s$ to the non-marked triangle numbered $s$. Number the marked triangles giving rise to $S$ from $1$ to $n$, let $H_1,\hdots,H_n$ be disjoint horocyclic ideal triangles each with compact side length $2r_{\chi,n}^k$, and for each $t$ fix a homeomorphism from $H_t$ to the complement, in the $t^{\mathrm{th}}$ marked triangle, of its marked vertex.

We now form a triangulated complex $F$ as a quotient space of $\left( \bigsqcup T_s\right)\sqcup\left(\bigsqcup H_t\right)$ by identifying edges of the $T_s$ and $H_t$ in pairs. For each $s$ and $t$ such that the images of $T_s$ and $T_t$ (or $T_s$ and $H_t$, or $H_s$ and $T_t$) share an edge, identify the corresponding edges of the geometric triangles by an isometry, choosing the one that is isotopic to the homeomorphism that pairs their images in the edge-pairing of marked and non-marked triangles that has quotient $S$.  Also isometrically identify the two non-compact edges of each $H_t$.  Since $S$ is closed, and each of its marked vertices lies in a single triangle (hence also a single edge), this pairs off all edges of the $T_s$ and $H_t$ in $F$. And our choices of edge-pairings ensure that the homeomorphism from $\left( \bigsqcup T_s\right)\sqcup\left(\bigsqcup H_t\right)$ to the corresponding collection of triangles for $S$ (less $n$ vertices) can be adjusted by an isotopy to induce a homeomorphism $F\to S-\mathcal{P}$. Its inverse is the map $f$ from the Proposition's statement. 

A standard argument now shows that $F$ is a hyperbolic surface by describing a family of chart maps to $\mathbb{H}^2$ with isometric transition functions.  This argument is essentially that of, say, \cite[Th.~10.1.2]{Ratcliffe}, so we will only give a bare sketch of ideas.  The quotient map $\left( \bigsqcup T_s\right)\cup\left(\bigsqcup H_t\right)\to F$ has a well-defined inverse on the interior of each triangle, and this yields charts for points that lie outside the triangulation's one-skeleton.  For a point $p$ in the interior of an edge of intersection between the images of, say, $T_s$ and $T_s$, there is a chart that maps $p$ to its preimage $\tilde{p}$ in $T_i$.  It sends the intersection of a neighborhood of $p$ with the image of $T_s$ into an isometric translate of $T_t$ that intersects $T_s$ along the edge containing $\tilde{p}$.

Similarly, for each vertex $p$ of the triangulation, a chart around $p$ is given by choosing a preimage $\tilde{p}$ of $p$ in some triangle whose image contains $p$, then isometrically translating all other triangles whose images contain $p$ so that they have a vertex at $\tilde{p}$.  The idea is to choose these isometries so that each translate intersects the translate of the triangle before it (as their images are encountered, proceeding around the boundary of a small neighborhood $U$ of $p$ in $F$) in an edge containing $\tilde{p}$.  It is key that by the definition of $r_{\chi,n}^k$ in Proposition \ref{Vor bound} we have
$$ \left(6-\frac{6\chi+3n}{k}\right)\alpha(r_{\chi,n}^k) + \frac{2n}{k}\beta(r_{\chi,n}^k) = 2\pi. $$
And since $p$ lies in $6-\frac{6\chi+3n}{k}$ equilateral triangle vertices and $2n/k$ horocyclic ideal triangle vertices by hypothesis, upon proceeding all the way around the boundary of $U$ we find that the union of isometric translates entirely encloses a neighborhood of $\tilde{p}$ in $\mathbb{H}^2$ that is isometric to $U$.

For each horocyclic ideal triangle $H_t$, the isometry that identifies the two non-compact edges of $H_t$ is a parabolic fixing its ideal point $u$.  This implies that each cross-section of $H_t$ by a horocycle with ideal point $u$ has its endpoints identified, so Theorem 11.1.4 of \cite{Ratcliffe} implies that $F$ is complete. Proposition \ref{Vor bound} now implies that $F$ has a packing by $k$ disks of radius $r_{\chi,n}^k$ centered at the vertices of the triangulation of $F$.  

Any two equilateral triangles in $\mathbb{H}^2$ with the same side length are isometric, and the full combinatorial symmetry group of any equilateral triangle is realized by isometries (these facts are standard). Analogously, two horocyclic ideal triangles with the same compact side length are isometric (this is easy to prove bare hands, or cf.~\cite[Prop.~3.4]{DeB_cyclic}), and every one has a reflection exchanging its two vertices in $\mathbb{H}^2$. It follows that if $f'\co S-\mathcal{P}\to F'$ has the same properties as $F$ then an isometry $\phi\co F'\to F$ can be defined triangle-by-triangle so that $f$ and $\phi\circ f'$ take each vertex, edge, and triangle of $S$ to identical corresponding objects in $F$, and that their restrictions to any edge are properly isotopic through maps to an edge of $F$. Adjusting $\phi\circ f'$ further on each triangle yields a proper isotopy to $f$.\end{proof}

\subsection{Constructing triangulated surfaces}\label{constrxn}  In this subsection we construct surfaces with prescribed triangulations to prove Proposition \ref{all triangulations}. We will treat the orientable and non-orientable cases separately, and it will be useful at times to think in terms of genus rather than Euler characteristic.  Here the \textit{genus} of an orientable (or, respectively, non-orientable) closed surface is the number of summands in a decomposition as a connected sum of tori (resp.~projective planes).  We declare the genus of a compact surface with boundary to be that of the closed surface obtained by adjoining a disk to each boundary component.  We now recall the fundamental relationship between the genus, number of boundary components, and Euler characteristic:\begin{align}\label{g vs x vs b}
	& \chi(F_{\mathit{or}}) = 2 - 2g_{\mathit{or}} - b & & \chi(F_{\mathit{non}}) = 2 - g_{\mathit{non}} - b \end{align}
On the left side above, $F_{\mathit{or}}$ is a compact, orientable surface of genus $g_{\mathit{or}}\geq 0$, and on the right, $F_{\mathit{non}}$ is non-orientable of genus $g_{\mathit{non}}\geq 1$, with $b\geq 0$ boundary components in each case.

In proving Proposition \ref{all triangulations}, we will find it convenient to track the ``triangle valence" of vertices.

\begin{notation}  We take the \textit{triangle valence} of a vertex $v$ of a triangulated surface to be the number of triangle vertices identified at $v$.\end{notation}

Since the link of a vertex $v$ in a closed triangulated surface is a circle, the triangle valence of $v$ coincides with its \textit{valence} as usually defined: the number of edge endpoints at $v$.  However for a vertex on the boundary of a triangulated surface with boundary, the triangle valence is one less than the valence.  It is convenient to track triangle valence since it is additive under the operation of identifying surfaces with boundary along their boundaries.

The proof is a bit lengthy and technical, though completely elementary, so before embarking on it we give an overview.  We build every triangulated surface by identifying boundary components in pairs from a fixed collection of ``building blocks'' constructed in a sequence of Examples.  Each building block is a compact surface with boundary which is triangulated with all non-marked vertices on the boundary, an equal number of vertices per boundary component, and all (non-marked) vertices of equal valence. To give an idea of what we will construct, we have collected data on our orientable building blocks without marked vertices in Table \ref{orientable props}.

\begin{table}[ht]
\begin{tabular}{c | l l l l l l l}
  \# vertices per  & $\Sigma_{0,2}$ & $\Sigma_{0,3}$ & $\Sigma_{0,4}$ & $\Sigma_{0,6}$ & 
	$\Sigma_{g,1},\ g\ge 1$ & $\Sigma_{g,2},\ g\ge 1$ & $\Sigma_{1,b}$ \\
  $\partial$-component & \multicolumn{4}{c}{(Example \ref{holed spheres})} & 
	\multicolumn{1}{c}{(\ref{one-hole block})} & \multicolumn{1}{c}{(\ref{two-hole block})} &
	(\ref{interpolators}) \\ \hline
  $1$ & $3$ & & & & $12g-3$ & $6g+3$ & $9$  \\
  $2$ & $3$ & $4$ & & $5$ & $6g$ & $3g+3$ & $6$ \\
  $3$ & $3$ & & $4$ & & $4g+1$ & $2g+3$ & $5$ \\ 
  \multicolumn{8}{c}{} \end{tabular}
\caption{Triangle valence of the orientable building blocks.}
\label{orientable props}
\end{table}

In the Table, columns correspond to the triangulated building blocks $\Sigma_{g,b}$ (of orientable genus $g$ with $b$ boundary components), rows to the number of vertices per boundary component, and table entries to the triangle valence of each vertex.  And the number in parentheses directly below each $\Sigma_{g,b}$ refers to the Example where it is triangulated.

We use these building blocks in Lemma \ref{orientable closed} to prove the orientable closed (i.e.~without marked vertices) case of Proposition \ref{all triangulations}. We then proceed to the non-orientable closed case, in Lemma \ref{nonorientable closed}, after adding a few non-orientable building blocks to the mix in Examples \ref{holed RP2} and \ref{non-orientable block}. Each Lemma's proof has several cases, featuring different combinations of building blocks, determined by certain divisibility conditions on the total number of vertices $k$.

As we remarked in the introduction, the closed case of Proposition \ref{all triangulations} is the $p=3$ case of the main theorem of Edmonds--Ewing--Kulkarni \cite{EdEwKu}, which is proved by a different method involving branched covers.  The advantage of our proof is that each closed surface constructed in Lemmas \ref{orientable closed} and \ref{nonorientable closed} has a collection of disjoint simple closed curves, coming from the building blocks' boundaries, which are unions of edges and whose union contains every vertex. This allows us to extend to the case of $n>0$ marked vertices (a case not covered in \cite{EdEwKu}) by ``unzipping'' each edge in each such curve and inserting a copy of a final building block constructed in Example \ref{marked block}, homeomorphic to a disk, with marked vertices in its interior. We handle this case below that Example, completing the proof of Proposition \ref{all triangulations}.

\begin{example}\label{holed spheres}  It is a simple exercise to show that an annulus $\Sigma_{0,2}$ can be triangulated with one, two, or three vertices on each boundary component, and each such triangulation can be arranged so that each vertex has triangle valence three.  We will also use three- and six-holed spheres $\Sigma_{0,3}$ and $\Sigma_{0,6}$ triangulated with two vertices per boundary component, and a four-holed sphere $\Sigma_{0,4}$ with three vertices per boundary component.  The three- and four-holed spheres are pictured in Figure \ref{holed sphere} on the left and in the middle, respectively.  Inspection of each reveals that each vertex has triangle valence four.

The right side of Figure \ref{holed sphere} shows a triangulated dodecagon, two copies of which are identified along alternating edges to produce a triangulated copy of $\Sigma_{0,6}$.  Precisely, for each even $i$ we identify the edge $e_i$ in one copy homeomorphically with $e_{i+6}$ in the other so that for each $i$, $v_i = e_i\cap e_{i+1}$ in the first copy is identified with $v_{i+6}$ in the other.  Note that for each $i \cong0$ modulo four, the vertex $v_i$ is contained in $4$ triangle vertices of the dodecagon, whereas $v_i$ is in $3$, $1$, or $2$ vertices for $i\cong1$, $2$, or $3$, respectively.  Thus in the quotient six-holed sphere $\Sigma_{0,6}$, each vertex has triangle valence $5$.  Each boundary component is the union of two copies of $e_i$ along their endpoints for some odd $i$ and so contains two vertex quotients, of $v_{i-1}$ and $v_i$.\end{example}

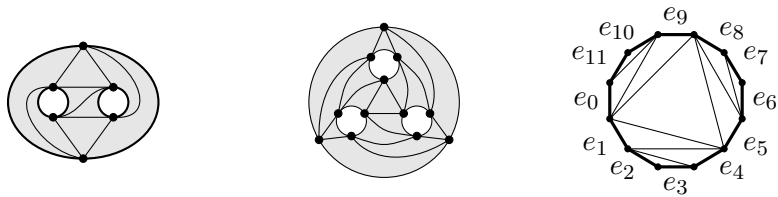
\begin{figure}
\begin{tikzpicture}

\begin{scope}[xshift=-4cm,scale=0.5]
	\draw [fill, opacity=0.1] (0,0) ellipse (2cm and 1.5cm);
	\draw [thick] (0,0) ellipse (2cm and 1.5cm);
	\draw [fill, color=white] (0.8,0) circle (0.4cm);
	\draw [thick] (0.8,0) circle (0.4cm);
	\draw [fill, color=white] (-0.8,0) circle (0.4cm);
	\draw [thick] (-0.8,0) circle (0.4cm);
	
	\draw [fill] (0,1.5) circle (0.1cm);
	\draw [fill] (0,-1.5) circle (0.1cm);
	\draw [fill] (0.8,0.4) circle (0.1cm);
	\draw [fill] (0.8,-0.4) circle (0.1cm);
	\draw [fill] (-0.8,0.4) circle (0.1cm);
	\draw [fill] (-0.8,-0.4) circle (0.1cm);
	
	\draw (-0.8,0.4) -- (0.8,0.4) -- (0,1.5) -- cycle;
	\draw (-0.8,-0.4) -- (0.8,-0.4) -- (0,-1.5) -- cycle;
	\draw (-0.8,-0.4) .. controls (0,-0.4) and (0,0.4) .. (0.8,0.4);
	\draw (0.8,-0.4) .. controls (2.2,-0.5) and (1.3,1.3) .. (0,1.5);
	\draw (-0.8,0.4) .. controls (-2.2,0.5) and (-1.3,-1.3) .. (0,-1.5);
\end{scope}

\begin{scope}[scale=0.5]
	\draw [fill, opacity=0.1] (0,0) circle [radius=2];
	\draw (0,0) circle [radius=2];
	\draw [fill] (0,2) circle [radius=0.1];
	\draw [fill] (-1.73,-1) circle [radius=0.1];
	\draw [fill] (1.73,-1) circle [radius=0.1];

	\draw [fill, color=white] (0,1) circle [radius=0.4];
	\draw (0,1) circle [radius=0.4];
	\draw [fill] (0,0.6) circle [radius=0.1];
	\draw [fill] (0.35,1.2) circle [radius=0.1];
	\draw [fill] (-0.35,1.2) circle [radius=0.1];

\draw [fill, color=white] (-0.87,-0.5) circle [radius=0.4];
\draw (-0.87,-0.5) circle [radius=0.4];
\draw [fill] (-0.87,-0.9) circle [radius=0.1];
\draw [fill] (-0.52,-0.3) circle [radius=0.1];
\draw [fill] (-1.22,-0.3) circle [radius=0.1];

\draw [fill, color=white] (0.87,-0.5) circle [radius=0.4];
\draw (0.87,-0.5) circle [radius=0.4];
\draw [fill] (0.87,-0.9) circle [radius=0.1];
\draw [fill] (0.52,-0.3) circle [radius=0.1];
\draw [fill] (1.22,-0.3) circle [radius=0.1];

\draw (-0.52,-0.3) -- (0.52,-0.3) -- (0,0.6) -- cycle;

\draw (-0.35,1.2) -- (0,2) -- (0.35,1.2);
\draw (-1.22,-0.3) -- (-1.73,-1) -- (-0.87,-0.9);
\draw (1.22,-0.3) -- (1.73,-1) -- (0.87,-0.9);

\draw (1.22,-0.3) .. controls (1.04,0.6) .. (0.35,1.2);
\draw (-1.22,-0.3) .. controls (-1.04,0.6) .. (-0.35,1.2);
\draw (-0.87,-0.9) .. controls (0,-1.2) .. (0.87,-0.9);

\draw (1.22,-0.3) .. controls (1.5,0.3) and (1.3,1.3) .. (0,2);
\draw (-0.35,1.2) .. controls (-1.01,1.15) and (-1.78,0.48) .. (-1.73,-1);
\draw (-0.87,-0.9) .. controls (-0.49,-1.45) and (0.48,-1.78) .. (1.73,-1);

\draw (-0.52,-0.3) .. controls (0,-0.8) .. (0.87,-0.9);
\draw (0.52,-0.3) .. controls (0.69,0.4) .. (0.35,1.2);
\draw (0,0.6) .. controls (-0.69,0.4) .. (-1.22,-0.3);
\end{scope}

\begin{scope}[xshift=3cm,yshift=0.9cm,scale=0.4]
    \draw [very thick] (0,-1.6) -- (0,-2.8) -- (0.6,-3.8) -- (1.6,-4.4) -- (2.8,-4.4) -- (3.8,-3.8) -- (4.4,-2.8);
    \draw [very thick] (4.4,-2.8) -- (4.4,-1.6) -- (3.8,-0.6) -- (2.8,0) -- (1.6,0) -- (0.6,-0.6) -- (0,-1.6);
    \draw [fill] (0,-1.6) circle (0.1cm);
    \draw [fill] (0,-2.8) circle (0.1cm);
    \draw [fill] (1.6,-4.4) circle (0.1cm);
    \draw [fill] (3.8,-3.8) circle (0.1cm);
    \draw [fill] (3.8,-0.6) circle (0.1cm);
    \draw [fill] (1.6,0) circle (0.1cm);
    \draw [fill] (0.6,-3.8) circle [radius=0.1];
    \draw [fill] (2.8,-4.4) circle [radius=0.1];
    \draw [fill] (4.4,-2.8) circle [radius=0.1];
    \draw [fill] (4.4,-1.6) circle [radius=0.1];
    \draw [fill] (2.8,0) circle [radius=0.1];
    \draw [fill] (0.6,-0.6) circle [radius=0.1];
    
    \draw (0,-2.8) -- (3.8,-3.8) -- (0.6,-3.8) -- (2.8,-4.4);
    \draw (3.8,-3.8) -- (2.8,0) -- (4.4,-2.8) -- (3.8,-0.6);
    \draw (2.8,0) -- (0,-2.8) -- (1.6,0) -- (0,-1.6);
    
    \node [left] at (0,-2.2) {$e_0$};
    \node [below left] at (0.3,-3.1) {$e_1$};
    \node [below left] at (1.2,-3.9) {$e_2$};
    \node [below] at (2.2,-4.4) {$e_3$};
    \node [below right] at (3.3,-3.9) {$e_4$};
    \node [below right] at (4.1,-3.1) {$e_5$};
    \node [right] at (4.4,-2.2) {$e_6$};
    \node [above right] at (4.1,-1.3) {$e_7$};
    \node [above right] at (3.3,-0.5) {$e_8$};
    \node [above] at (2.2,0) {$e_9$};
    \node [above left] at (1.2,-0.5) {$e_{10}$};
    \node [above left] at (0.3,-1.3) {$e_{11}$};
\end{scope}

\end{tikzpicture}
\caption{The triangulated three-,  four-, and (half the) six-holed sphere, left-to-right.}
\label{holed sphere}
\end{figure}

\begin{example}\label{one-hole block}  Here we will triangulate the orientable genus-$g$ surface $\Sigma_{g,1}$ with one boundary component, for $g\geq 1$.  In fact we describe triangulations with one, two, and three vertices, all on the boundary and all with the same valence.

The standard construction of the closed, orientable genus-$g$ surface takes a $4g$-gon $P_{4g}$ with edges labeled $e_0,\hdots,e_{4g-1}$ in counterclockwise order, and identifies $e_i$ to $e_{i+2}$ via an orientation-reversing homeomorphism for each $i<4g$ congruent to $0$ or $1$ modulo $4$.  (Here the $e_i$ inherit their orientation from the counterclockwise orientation on $\partial P_{4g}$.)  We produce $\Sigma_{g,1}$ by identifying the first $4g$ edges of a $(4g+1)$-gon $P_{4g+1}$ in the same way, and making no nontrivial identifications on points in the interior of the final edge.  Any triangulation of $P_{4g+1}$ projects to a one-vertex triangulation of $\Sigma_{g,1}$ with its single vertex $v$ on $\partial \Sigma_{g,1}$.  Such a triangulation has $4g-1$ triangles, so $v$ has triangle valence $12g-3$.  The case $g=1$ is pictured on the left in Figure \ref{one-hole fig}.

\newcommand\pentagon{
\draw [directed] (0,-1) -- (0.9,-1.3);
\draw [ddirected] (0.9,-1.3) -- (1.5,-0.5);
\draw [reverse directed] (1.5,-0.5) -- (0.9,0.3);
\draw [reverse ddirected] (0.9,0.3) -- (0,0);
\draw [fill] (0,0) circle [radius=0.06];
\draw [fill] (0,-1) circle [radius=0.06];
\draw [fill] (0.9,-1.3) circle [radius=0.06];
\draw [fill] (1.5,-0.5) circle [radius=0.06];
\draw [fill] (0.9,0.3) circle [radius=0.06]; }

\begin{figure}[ht]
\begin{tikzpicture}

\begin{scope}[xshift=-5cm]
\pentagon
\draw [thick] (0,0) -- (0,-1);
\draw (0,0) -- (1.5,-0.5) -- (0,-1);
\node [below left] at (0.5,-1.1) {$e_0$};
\node [below right] at (1.1,-0.8) {$e_1$};
\node [above right] at (1.2,-0.1) {$e_2$};
\node [above left] at (0.5,0.1) {$e_3$};
\end{scope}

\begin{scope}
\pentagon
\draw [thick] (0,0) -- (0.2,-0.5) -- (0,-1);
\draw (0.9,0.3) -- (0.2,-0.5) -- (1.5,-0.5);
\draw (0.2,-0.5) -- (0.9,-1.3);
\draw (0.2,-0.5) circle [radius=0.06];
\end{scope}

\begin{scope}[xshift=5cm]
\pentagon
\draw [thick] (0,0) -- (0.2,-0.3) -- (0.2,-0.7) -- (0,-1);
\draw (0.9,0.3) -- (0.2,-0.3) -- (1.5,-0.5) -- (0.2,-0.7) -- (0.9,-1.3);
\draw (0.2,-0.3) circle [radius=0.06];
\draw (0.2,-0.7) circle [radius=0.06];
\end{scope}

\end{tikzpicture}
\caption{Triangulations of the one-holed torus with one, two or three vertices.}
\label{one-hole fig}
\end{figure}
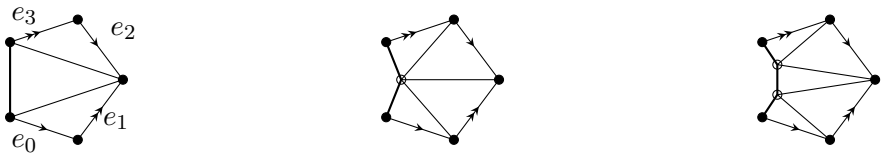

We may construct two- or three-vertex triangulations of $\Sigma_{g,1}$, each with all vertices on the boundary, by inserting one or two vertices, respectively, in the interior of the non-identified edge $e_{4g}$ of $P_{4g+1}$.  If one vertex is inserted then we begin by joining it to each vertex $e_{i-1}\cap e_{i}$ for $1\leq i \leq 4g-1$.  If two are inserted then the nearer one to $e_0\cap e_{4g}$ is joined to $e_{i-1}\cap e_i$ for $1\leq i \leq 2g$, and the other one is joined to $e_{i-1}\cap e_{i}$ for $2g\leq i \leq 4g-1$.  See the middle and right side of Figure \ref{one-hole fig} for the case $g=1$.  Note that in the resulting triangulations of $\Sigma_{g,1}$, the inserted vertices have lower valence than the quotient vertex $v$ of those of $P_{4g+1}$.  

We can even out the valence by \textit{flipping} edges.  An edge $e$ that is the intersection of distinct triangles $T$ and $T'$ of a triangulated surface is flipped by replacing it with the other diagonal of the quadrilateral $T\cup T'$.  This yields a new triangulation in which each endpoint of $e$ has its (triangle) valence reduced by by one, and each vertex opposite $e$ has it increased by one.  In $\Sigma_{g,1}$, triangulated as prescribed in the paragraph above, the projection of each $e_i$ begins and ends at $v$, and each vertex opposite $e_i$ is an inserted vertex.  So flipping $e_i$ reduces the valence of $v$ and increases the valence of the inserted vertex, each by $2$

In the the two-vertex triangulation of $\Sigma_{g,1}$ described above, the inserted vertex has triangle valence $4g$.  Since there are a total of $4g$ triangles there are $12g$ triangle vertices total.  So after flipping $e_i$ for $g$ distinct $i$, each vertex of $A_{4g}$ has triangle valence $6g$.  In the three-vertex triangulation, each inserted vertex has triangle valence $2g+1$, and there are $3(4g+1)$ triangle vertices total.  So after flipping all $2g$ distinct $e_i$, all vertices have triangle valence $4g+1$.
\end{example}

\begin{example}\label{two-hole block} For each $g\geq 1$ we now triangulate the two-holed orientable surface $\Sigma_{g,2}$ of genus $g$ with two boundary components.  Each triangulation will have all vertices on the boundary, each boundary component will have the same number of vertices (either one, two, or three) and each vertex will have the same valence.  We construct $\Sigma_{g,2}$ by identifying all but two of the edges of a $4n$-gon $P_{4n}$ in pairs, where $n = g+1$.

Label the edges of $P_{4n}$ as $e_0,\hdots,e_{4n-1}$ in counterclockwise order, and for each $k\neq 0,2n$ identify $e_k$ with its diametrically opposite edge $e_{k+2n}$ by an orientation-reversing homeomorphism.  The edge orientations in question here are inherited from the boundary orientation on $\partial P_{4n}$.  So the initial vertex $v_1 = e_0\cap e_1$ of $e_1$ is identified with the terminal vertex $v_{2n+2}$ of $e_{2n+1}$.  Since $v_{2n+2}$ is the initial vertex of $e_{2n+2}$ it is also identified with the terminal vertex $v_3$ of $e_2$ and so on, so that in the end all vertices $v_i$ for odd $i<2n$ are identified with all vertices $v_j$ for even $j > 2n$.  In particular the endpoints $v_0 = v_{4n}$ and $v_1$ of $e_0$ are identified in the quotient.

Similarly, for all even $i$ with $0<i\leq 2n$, the vertices $v_i$ are identified together with the vertices $v_j$ for all odd $j > 2n$; and in particular the endpoints of $e_{2n}$ are identified in the quotient.  The quotient $\Sigma_{g,2}$ by these identifications is thus a surface with two boundary components, one from $e_0$ and one from $e_{2n}$, each containing one of the two equivalence classes of vertices of $P_{4n}$.  Note that the $180$-degree rotation of $P_{4n}$ preserves the identifications and so induces an automorphism $\rho$ of $\Sigma_{g,2}$ that exchanges its two boundary components and the vertex quotients they contain.

\newcommand{\oct}[2]{
  \begin{scope}[xshift=#1cm,yshift=#2cm,scale=0.9]
  \draw [very thick] (0,-3.4) -- (1.4,-4.8) -- (3.4,-4.8) -- (4.8,-3.4);
  \draw [very thick] (4.8,-1.4) -- (3.4,0) -- (1.4,0) -- (0,-1.4);
  \draw [fill] (-.1,-1.5) rectangle (.1,-1.3);
  \draw [fill] (-.1,-3.5) rectangle (.1,-3.3);
  \draw [fill] (3.3,-.1) rectangle (3.5,.1);
  \draw [fill] (3.3,-4.9) rectangle (3.5,-4.7);
  \draw [fill] (1.4,0) circle [radius=0.1];
  \draw [fill] (1.4,-4.8) circle [radius=0.1];
  \draw [fill] (4.8,-1.4) circle [radius=0.1];
  \draw [fill] (4.8,-3.4) circle [radius=0.1];
  \end{scope}
}

\newcommand{\dodec}[2]{
  \begin{scope}[xshift=#1cm,yshift=#2cm]
    \draw [very thick] (0,-2.8) -- (0.6,-3.8) -- (1.6,-4.4) -- (2.8,-4.4) -- (3.8,-3.8) -- (4.4,-2.8);
    \draw [very thick] (4.4,-1.6) -- (3.8,-0.6) -- (2.8,0) -- (1.6,0) -- (0.6,-0.6) -- (0,-1.6);
    \draw [fill] (-0.1,-1.7) rectangle (0.1,-1.5);
    \draw [fill] (-0.1,-2.9) rectangle (0.1,-2.7);
    \draw [fill] (1.5,-4.5) rectangle (1.7,-4.3);
    \draw [fill] (3.7,-3.9) rectangle (3.9,-3.7);
    \draw [fill] (3.7,-0.7) rectangle (3.9,-0.5);
    \draw [fill] (1.5,-0.1) rectangle (1.7,0.1);
    \draw [fill] (0.6,-3.8) circle [radius=0.1];
    \draw [fill] (2.8,-4.4) circle [radius=0.1];
    \draw [fill] (4.4,-2.8) circle [radius=0.1];
    \draw [fill] (4.4,-1.6) circle [radius=0.1];
    \draw [fill] (2.8,0) circle [radius=0.1];
    \draw [fill] (0.6,-0.6) circle [radius=0.1];
  \end{scope}
}

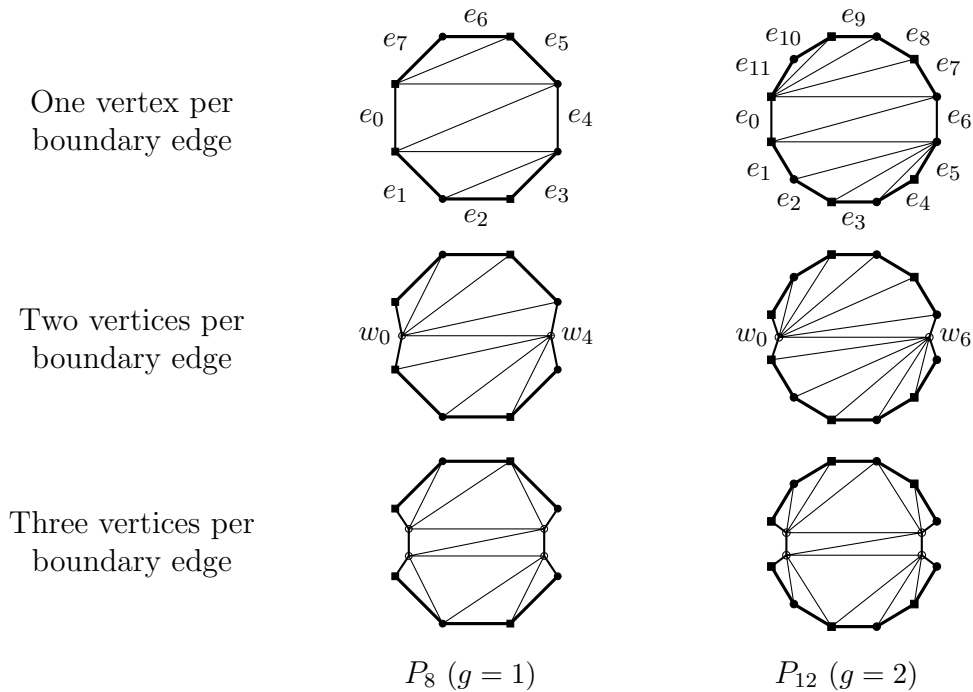
\begin{figure}
\begin{tikzpicture}[scale=0.5]

\node at (-7,-1.8) {\large One vertex per};
\node at (-7,-2.8) {\large boundary edge};

  \oct{0}{0};
  \begin{scope}[xshift=0cm,yshift=0cm,scale=0.9]
    \draw [thick] (0,-1.4) -- (0,-3.4);
    \draw [thick] (4.8,-3.4) -- (4.8,-1.4);
    \draw (3.4,0) -- (0,-1.4) -- (4.8,-1.4) -- (0,-3.4) -- (4.8,-3.4) -- (1.4,-4.8);
    \node [left] at (0,-2.4) {$e_0$};
    \node [below left] at (0.7,-4.1) {$e_1$};
    \node [below] at (2.4,-4.8) {$e_2$};
    \node [below right] at (4.1,-4.1) {$e_3$};
    \node [right] at (4.8,-2.4) {$e_4$};
    \node [above right] at (4.1,-0.7) {$e_5$};
    \node [above] at (2.4,0) {$e_6$};
    \node [above left] at (0.7,-0.7) {$e_7$};
  \end{scope}
  
  \dodec{10}{0};
  \begin{scope}[xshift=10cm,yshift=0cm]
    \draw [thick] (0,-1.6) -- (0,-2.8);
    \draw [thick] (4.4,-2.8) -- (4.4,-1.6);
    \draw (1.6,0) -- (0,-1.6) -- (2.8,0);
    \draw (3.8,-0.6) -- (0,-1.6) -- (4.4,-1.6) -- (0,-2.8) -- (4.4,-2.8) -- (0.6,-3.8);
    \draw (1.6,-4.4) -- (4.4,-2.8) -- (2.8,-4.4);
    \node [left] at (0,-2.2) {$e_0$};
    \node [below left] at (0.3,-3.1) {$e_1$};
    \node [below left] at (1.1,-3.9) {$e_2$};
    \node [below] at (2.2,-4.4) {$e_3$};
    \node [below right] at (3.3,-3.9) {$e_4$};
    \node [below right] at (4.1,-3.1) {$e_5$};
    \node [right] at (4.4,-2.2) {$e_6$};
    \node [above right] at (4.1,-1.3) {$e_7$};
    \node [above right] at (3.3,-0.5) {$e_8$};
    \node [above] at (2.2,0) {$e_9$};
    \node [above left] at (1.1,-0.5) {$e_{10}$};
    \node [above left] at (0.3,-1.3) {$e_{11}$};
  \end{scope}

\node at (-7,-7.6) {\large Two vertices per};
\node at (-7,-8.6) {\large boundary edge};

  \oct{0}{-5.8};
  \begin{scope}[xshift=0cm,yshift=-5.8cm,scale=0.9]
    \draw [thick] (0,-1.4) -- (0.2,-2.4) -- (0,-3.4);
    \draw [thick] (4.8,-3.4) -- (4.6,-2.4) -- (4.8,-1.4);
    \draw (0.2,-2.4) circle [radius=0.1];
    \node [left] at (0.2,-2.4) {$w_0$};
    \draw (4.6,-2.4) circle [radius=0.1];
    \node [right] at (4.6,-2.4) {$w_4$};
    \draw (3.4,0) -- (0.2,-2.4) -- (1.4,0);
    \draw (4.8,-1.4) -- (0.2,-2.4) -- (4.6,-2.4) -- (0,-3.4);
    \draw (1.4,-4.8) -- (4.6,-2.4) -- (3.4,-4.8);
  \end{scope}

  \dodec{10}{-5.8};
  \begin{scope}[xshift=10cm,yshift=-5.8cm]
    \draw [thick] (0,-1.6) -- (0.2,-2.2) -- (0,-2.8);
    \draw [thick] (4.4,-2.8) -- (4.2,-2.2) -- (4.4,-1.6);
    \draw (0.2,-2.2) circle [radius=0.1];
    \node [left] at (0.2,-2.2) {$w_0$};
    \draw (4.2,-2.2) circle [radius=0.1];
    \node [right] at (4.2,-2.2) {$w_6$};
    \draw (0.6,-0.6) -- (0.2,-2.2) -- (1.6,0);
    \draw (2.8,0) -- (0.2,-2.2) -- (3.8,-0.6);
    \draw (4.4,-1.6) -- (0.2,-2.2) -- (4.2,-2.2) -- (0,-2.8);
    \draw (0.6,-3.8) -- (4.2,-2.2) -- (1.6,-4.4);
    \draw (2.8,-4.4) -- (4.2,-2.2) -- (3.8,-3.8);
  \end{scope}

\node at (-7,-13) {\large Three vertices per};
\node at (-7,-14) {\large boundary edge};

  \oct{0}{-11.3};
  \begin{scope}[xshift=0cm,yshift=-11.3cm,scale=0.9]
    \draw [thick] (0,-1.4) -- (0.4,-2) -- (0.4,-2.8) -- (0,-3.4);
    \draw [thick] (4.8,-1.4) -- (4.4,-2) -- (4.4,-2.8) -- (4.8,-3.4);
    \draw (0.4,-2) circle [radius=0.1];
    \draw (0.4,-2.8) circle [radius=0.1];
    \draw (4.4,-2) circle [radius=0.1];
    \draw (4.4,-2.8) circle [radius=0.1];
    \draw (1.4,0) -- (0.4,-2) -- (3.4,0) -- (4.4,-2);
    \draw (0.4,-2) -- (4.4,-2) -- (0.4,-2.8) -- (4.4,-2.8) -- (3.4,-4.8);
    \draw (4.4,-2.8) -- (1.4,-4.8) -- (0.4,-2.8);
  \end{scope}

  \dodec{10}{-11.3};
  \begin{scope}[xshift=10cm,yshift=-11.3cm]
    \draw [thick] (0,-1.6) -- (0.4,-1.9) -- (0.4,-2.5) -- (0,-2.8);
    \draw [thick] (4.4,-1.6) -- (4,-1.9) -- (4,-2.5) -- (4.4,-2.8);
    \draw (0.4,-1.9) circle [radius=0.1];
    \draw (0.4,-2.5) circle [radius=0.1];
    \draw (4,-1.9) circle [radius=0.1];
    \draw (4,-2.5) circle [radius=0.1];
    \draw (0.6,-0.6) -- (0.4,-1.9) -- (1.6,0);
    \draw (2.8,0) -- (4,-1.9) -- (3.8,-0.6);
    \draw (2.8,0) -- (0.4,-1.9) -- (4,-1.9) -- (0.4,-2.5) -- (4,-2.5) -- (1.6,-4.4);
    \draw (0.6,-3.8) -- (0.4,-2.5) -- (1.6,-4.4);
    \draw (2.8,-4.4) -- (4,-2.5) -- (3.8,-3.8);
  \end{scope}

\node at (2,-17) {{\large $P_8$} ($g=1$)};
\node at (12,-17) {{\large $P_{12}$} ($g=2$)};

\end{tikzpicture}
\caption{First steps to triangulating $\Sigma_{g,2}$, for $g=1$ and $2$.}
\label{closed triangulations}
\end{figure}

We may triangulate $P_{4n}$ using arcs joining $v_0$ to $v_j$ for $2n+1\leq j < 4n-1$, $v_{2n}$ to $v_i$ for $1 \leq i < 2n-1$, and $v_1$ to $v_{2n+1}$.  The cases $g=1$ and $g=2$ (so $P_8$ and $P_{12}$, respectively) of this construction are pictured on the top line of Figure \ref{closed triangulations}.  The resulting triangulations of $\Sigma_{g,2}$ are $\rho$-invariant, since for example $\rho$ takes $v_1$ to $v_{2n+1}$, so the two vertices of $\Sigma_{g,2}$ have the same valence.  A triangulation of $P_{4n}$ has $4n-2$ triangles, so this is the number $f$ of faces of the triangulation of $\Sigma_{g,2}$.  There are two vertices, and the number $e$ of edges satisfies $2e-2 = 3f$ (note that the edges $e_0$ and $e_{2n}$ belong to only one triangle each), so $e = \frac{3}{2}f+1$. Therefore $\Sigma_{g,2}$ has Euler characteristic $2-2n$.  Since it has two boundary components its genus is $g$ as asserted.

We may re-triangulate $\Sigma_{g,2}$ with an additional one or two vertices per boundary component.  We first describe how to add one vertex, yielding a total of two vertices per boundary component.  We begin by adding vertices $w_0$ and $w_{2n}$ to $P_{4n}$ in the edges $e_0$ and $e_{2n}$, respectively.  Then triangulate the resulting $(4n+2)$-gon by joining $w_0$ to $w_{2n}$ and each $v_j$ for $2n < j \leq 4n-1$, and joining $w_{2n}$ to each $v_i$ for $1 \leq i < 2n$.  This is illustrated on the middle line of Figure \ref{closed triangulations}.  The resulting triangulation of $\Sigma_{g,2}$ is $\rho$-invariant, so the two quotient vertices of the $v_i$ have identical valence, as do the projections of $w_0$ and $w_{2n}$.

However it is plain to see that $w_0$ has triangle valence $2n+1$: it is one vertex of each of the $2n$ triangles in the upper half of $P_{4n}$, and one of a unique triangle in the lower half.  But since there are $4n$ triangles there are $12n$ triangle endpoints, so the two quotient vertices of the $v_i$ must each have triangle valence $4n-1$.  As in Example \ref{one-hole block}, we even out the valence by flipping some of the $e_i$.  For each $i$ between $1$ and $2n-1$, one of the triangles in $\Sigma_{g,2}$ containing $e_i$ has $w_0$ as its opposite vertex, and the other has $w_{2n}$ in this role.  If we flip $e_i$ we thus increases the triangle valence of each of $w_0$ and $w_{2n}$ by one, and since $e_i$ has one endpoint at each quotient vertex of the $v_i$, it decreases each of their triangle valences by one.  Thus after flipping $e_1$ through $e_{n-1}$, all vertices of the new triangulation of $\Sigma_{g,2}$  have triangle valence $3n$.

To re-triangulate $\Sigma_{g,2}$ with three vertices per boundary component, we begin by placing vertices $u_0$ and $w_0$ in the interior of $e_0$, in that order, and $u_{2n}$ and $w_{2n}$ in $e_{2n}$ so that the $180$-degree rotation of $P_{4n}$ exchanges $u_0$ with $u_{2n}$ and $w_0$ with $w_{2n}$.  Using line segments join $w_0$ to each of $v_2,\hdots,v_n$; join $u_{2n}$ to each of $v_{n+1},\hdots,v_{2n-1}$; join $w_{2n}$ to each of $v_{2n+2},\hdots,v_{3n}$; and join $u_0$ to each of $v_{3n+1},\hdots,v_{4n-1}$.  Note that the collection of such line segments is rotation-invariant.  It divides $P_{4n}$ into triangles and a single region with vertices $u_0$, $w_0$, $v_n$, $v_{n+1}$, $u_{2n}$, $w_{2n}$, $u_{3n}$ and $u_{3n+1}$.  Triangulate this region by joining $u_0$ to $w_{2n}$ and $v_{3n}$, $w_0$ to $w_{2n}$ and $u_{2n}$, and $v_n$ to $u_{2n}$.

The resulting triangulation of $P_{4n}$ is still rotation-invariant, and moreover each of $u_0$ and $w_0$ has triangle valence $n+2$.  There are a total of $4n+2$ triangles, with a total of $12n+6$ vertices.  So flipping $e_i$ for all $i$ between $0$ and $2n$ except $n$ yields a triangulation of $\Sigma_{g,2}$ in which each vertex has triangle valence $2n+1 = 2g+3$.
\end{example}

\begin{example}\label{interpolators}  In the special case $g=1$, the construction of Example \ref{one-hole block} yields a one-holed torus $\Sigma_{1,1}$ triangulated with one, two, or three vertices per boundary component such that each vertex has triangle valence $9$, $6$ or $5$, respectively.  We may construct a $b$-holed torus $\Sigma_{1,b}$ as a $b$-fold cover of $\Sigma_{1,1}$, where each boundary component projects homeomorphically to that of $\Sigma_{1,1}$.  One easily constructs such a cover by, say, joining a disk to $\Sigma_{1,1}$ along its boundary, taking a $b$-fold cover of the resulting torus, removing the preimage of the disk's interior and lifting a triangulation of $\Sigma_{1,1}$.  The triangle valence of vertices remains the same.\end{example}

We now have enough building blocks to handle the orientable closed case of Proposition \ref{all triangulations}.

\begin{lemma}\label{orientable closed}  For any $g\geq 2$ and any $k\in\mathbb{N}$ that divides $12(g-1)$, the closed, orientable surface of genus $g$ has a triangulation with $k$ vertices, all of equal valence.\end{lemma}

\begin{proof}  For the case $k=1$ we observe that any triangulation of the $4g$-gon $P_{4g}$ descends to a one-vertex triangulation of the genus-$g$ surface under the ``standard construction'' mentioned in Example \ref{one-hole block}.  So we assume below that $k>1$.   

By an Euler characteristic calculation, a $k$-vertex triangulation of the genus-$g$ surface  has $4(g-1) + 2k$ triangles.  If the vertices have equal valence then each has valence $\frac{12(g-1)}{k} + 6$.  Note that the number of vertices and their valence determines the genus, so below it is enough to exhibit $k$-vertex triangulations of closed, connected surfaces with vertices of the correct valence.  We break the proof into sub-cases depending on some congruence conditions satisfied by $k$.\begin{description}
  \item[Case 1: $2$ and $3$ do not divide $k$]  In this case $k$ has no common factor with $12$, so $k$ divides $g-1$.  For $g_0 = (g-1)/k$, we claim that the desired surface is obtained by joining one copy of the one-holed genus-$g_0$ surface $\Sigma_{g_0,1}$ of Example \ref{one-hole block} to each of the $k$ boundary components of $\Sigma_{1,k}$, where all building blocks are triangulated with one vertex per boundary component, so that vertices are identified.  
  
This follows from the fact that the vertex on $\Sigma_{g_0,1}$ has triangle valence $12g_0-3$ and each vertex of $\Sigma_{1,k}$ has triangle valence $9$; that triangle valence adds upon joining boundary components; and that triangle valence coincides with valence for vertices of triangulated closed surfaces.  So each vertex of the resulting closed surface has valence:
  $$ (12g_0-3)+9 = \frac{12(g-1)}{k} + 6 $$
  \item[Case 2: $2$ divides $k$, $3$ and $4$ do not]  Now $k/2$ has no common factor with $12$ and hence divides $g-1$.  So we use $k/2$ copies of $\Sigma_{g_0,2}$ from Example \ref{two-hole block}, where $g_0 = 2(g-1)/k$, each triangulated with one vertex on each boundary component.  We arrange them in a ring with a copy of the annulus $\Sigma_{0,2}$ placed between each pair of subsequent copies.  Each vertex of the resulting closed surface has valence:
  $$(6g_0+3)+3 = \frac{12(g-1)}{k} + 6$$
  \item[Case 3: $4$ divides $k$, $3$ does not]  In this case we use $k/4$ copies of $\Sigma_{g_0,2}$, where $g_0 = 4(g-1)/k$, each triangulated with two vertices on each boundary component.  As in the previous case we arrange them in a ring interspersed with copies of $\Sigma_{0,2}$, so after joining boundaries each vertex has valence $(3g_0+3) + 3 = \frac{12(g-1)}{k} + 6$.
  \item[Case 4: $3$ divides $k$, $2$ does not]  We take $g_0 = 3(g-1)/k\in\mathbb{N}$ and join $k/3$ copies of $\Sigma_{g_0,1}$, each triangulated with $3$ vertices per boundary component, to  $\Sigma_{1,k/3}$ with the corresponding triangulation to produce a closed triangulated surface.  Its vertices each have valence $(4g_0+1)+5 = \frac{12(g-1)}{k} + 6$.
  \item[Case 5: $6$ divides $k$, $4$ does not]  We take $g_0 = 6(g-1)/k\in\mathbb{N}$ and arrange $k/6$ copies of $\Sigma_{g_0,2}$, each triangulated with three vertices per boundary component, in a ring interspersed with copies of $\Sigma_{0,2}$.  The resulting closed, triangulated surface has vertices of valence $(2g_0+3)+3 = \frac{12(g-1)}{k} + 6$.
  \item[Case 6: $12$ divides $k$]  If $12(g-1)/k$ is even then we let $g_0 = 12(g-1)/(2k)$, so that
  $$ 2g_0+3 = \frac{12(g-1)}{k} + 3 $$
As in the previous case we join $k/6$ copies of $\Sigma_{g_0,2}$, triangulated with three vertices per boundary component, in a ring interspersed with copies of $\Sigma_{0,2}$.  If $12(g-1)/k$ is odd then we let $g_0 = (12(g-1)/k-1)/2$, so that
  $$ 2g_0+3 = \frac{12(g-1)}{k} + 2 $$
We build the surface in this case from $k/6$ copies of $\Sigma_{g_0,2}$, triangulated with three vertices per boundary component, and $k/12$ copies of $\Sigma_{0,4}$, each triangulated as in Figure \ref{holed sphere}. (If $g_0 = 0$, as it may be, then $\Sigma_{g_0,2}$ is the annulus of Example \ref{holed spheres}.) Given any bijection from the set of boundary components of of the $\Sigma_{g_0,2}$ to those of the $\Sigma_{0,4}$, homeomorphically identifying each boundary component of a $\Sigma_{g_0,2}$ with its image taking vertices to vertices produces a closed, triangulated surface. We must choose the bijection to make the resulting surface connected, an easy exercise equivalent to constructing a connected four-valent graph with $k/12$ vertices (treating the $\Sigma_{0,4}$ as vertices and the $\Sigma_{g_0,2}$ as edges).
\end{description}
Each case above constructs a closed, connected surface triangulated with $k$ vertices, each of valence $\frac{12(g-1)}{k} + 6$.  These quantities determine the number of edges and faces of the triangulation and show that the surface constructed has Euler characteristic $2-2g$, hence genus $g$.
\end{proof}

We now prove the same result in the non-orientable case.  Recall below that the \textit{genus} of a non-orientable surface is the maximal number of $\mathbb{R}P^2$-summands in a connected sum decomposition.  We begin by adding some non-orientable building blocks to the mix.

\begin{example}\label{holed RP2}  In Example \ref{holed spheres} the edge identifications between the two copies of the dodecagon comprising $\Sigma_{0,6}$ preserve orientation, so these copies inherit opposite orientations from any orientation on $\Sigma_{0,6}$.  Therefore the involution that exchanges the two dodecagons while rotating $\Sigma_{0,6}$ by $180$-degrees reverses orientation.  Since it is also triangulation-preserving and fixed point-free, its quotient is a three-holed $\mathbb{R}P^2$, triangulated with two vertices per boundary component where again each vertex has triangle valence $5$.
\end{example}

\begin{example}\label{non-orientable block}  For any $g\geq 2$ and a $2g$-gon with edges oriented counterclockwise and subsequently labeled $e_0,\hdots,e_{2g-1}$, identifying $e_i$ with $e_{i+1}$ by an orientation-preserving homeomorphism for each even $i<2n$ yields a non-orientable surface of genus $g$.  Any triangulation of the $2g$-gon projects to a one-vertex triangulation of this surface with $2g-2$ triangles.  For $g\geq 1$ a one-vertex triangulation of a one-holed genus-$g$ nonorientable surface $\Upsilon_{g,1}$ is produced by analogously by identifying all edges but one of a $2g+1$-gon.  This triangulation thus has $2g-1$ triangles, so the vertex has triangle valence $6g-3$.

We produce two- or three-vertex triangulations with vertices of constant valence analogously to the orientable case of Example \ref{one-hole block}.  These triangulations have $2g$ and $2g+1$ triangles, respectively, so each vertex has triangle valence $3g$ or $2g+1$.\end{example}


\begin{lemma}\label{nonorientable closed}  For any $g\geq 3$ and any $k\in\mathbb{N}$ that divides $6(g-2)$, the closed, nonorientable genus-$g$ surface has a triangulation with $k$ vertices, all of equal valence.\end{lemma}

\begin{proof}  The $k=1$ case is given at the beginning of Example \ref{non-orientable block} above.  For the case $k>1$, as in the proof of Lemma \ref{orientable closed} we consider several sub-cases.
\begin{description}  \item[Case 1: $2$ and $3$ do not divide $k$]  For $g_0 = (g-2)/k$ we glue one copy of $\Upsilon_{g_0,1}$, triangulated with one vertex per boundary component, to each boundary component of $\Sigma_{1,k}$, triangulated to match.  The resulting closed surface has each vertex of valence $(6g_0 -3)+9 = 6(g-2)/k + 6$.
\item[Case 2: $2$ divides $k$, $3$ does not]  For $g_0 = 2(g-2)/k$, join $k/2$ copies of $\Upsilon_{g_0,1}$ to boundary components of $\Sigma_{1,k/2}$, all triangulated with two vertices per boundary component.  In the resulting surface each vertex has valence $3g_0+6=6(g-2)/k + 6$.
\item[Case 3: $3$ divides $k$, $2$ does not]  For $g_0 = 3(g-2)/k$, join $k/3$ copies of $\Upsilon_{g_0,1}$ to boundary components of $\Sigma_{1,k/3}$, all triangulated with three vertices per boundary component.  In the resulting surface each vertex has valence $(2g_0+1) + 5=6(g-2)/k + 6$.
\item[Case 4: $6$ divides $k$] If $6(g-2)/k$ is even then take $g_0 = 6(g-2)/(2k)$ and join $k/3$ copies of $\Upsilon_{g_0,1}$ to boundary components of $\Sigma_{1,k/3}$, all triangulated with three vertices per boundary component.  In the resulting surface each vertex has valence $(2g_0+1)+5 = 6(g-2)/k+6$.  If $6(g-2)/k$ is odd and congruent to $0$ modulo three, take $g_0=6(g-2)/(3k)$ and join $k/2$ copies of $\Upsilon_{g_0,1}$ to boundary components of $\Sigma_{1,k/2}$, all triangulated with two vertices per boundary component.  In the resulting surface each vertex has valence $3g_0 + 6 = 6(g-2)/k+6$.

If $6(g-2)/k$ is odd and congruent to $1$ modulo three, take $g_0 = (6(g-2)/k+2)/3$ and take $g_1 = g_0-1$.  We join $k/6$ copies of $\Sigma_{0,3}$, triangulated with two vertices per boundary component as in Figure \ref{holed sphere}, to $k/6$ copies of each of $\Upsilon_{g_0,1}$ and the two-holed orientable building block $\Sigma_{g_1,2}$, triangulated to match, as follows: arrange the copies of $\Sigma_{g_1,2}$ in a ring and join boundary components of each pair of subsequent copies to two boundary components of a fixed copy of $\Sigma_{0,3}$.  This leaves one free boundary component on each copy of $\Sigma_{0,3}$, which we cap off with a copy of $\Upsilon_{g_0,1}$.  The result is a connected, closed triangulated surface with $k$ vertices, each of valence $3g_0+4 = 3(g_1+1)+4 = 6(g-2)/k+6$.

If $6(g-2)/k$ is odd and congruent to $2$ modulo three then we perform the same construction as in the previous case, except that we take $g_0 = [6(g-2)/k+1]/3$ and replace each copy of $\Sigma_{0,3}$ with a copy of the three-holed $\mathbb{R}P^2$ from Example \ref{holed RP2}.  The resulting closed, triangulated surface now has $k$ vertices that each have valence $3g_0+5 = 3(g_1+1) + 5 = 6(g-2)/k+6$.\end{description}
As in the proof of Lemma \ref{orientable closed}, the fact that each non-orientable surface constructed above has $k$ vertices, each of valence $6(g-2)/k+6$, implies that it has genus $g$.\end{proof}

\begin{example}[Triangulated complexes with ideal vertices]\label{marked block} Here we will produce a triangulated complex $X_l$ homeomorphic to a disk for each $l\in\mathbb{N}$, with $l+1$ vertices of which $l$ lie in the interior and each have triangle valence one.  We call these vertices ``marked''.   The remaining vertex lies on the boundary of $X_l$.  The triangulation of $X_l$ is comprised of $2l-1$ triangles in two classes: ``marked'' triangles $H_1,\hdots,H_l$, which each have one marked vertex, and ``non-marked'' triangles $E_1,\hdots,E_{l-1}$.

\newcommand\triangleone{
  \draw (0,0) -- (2,0) -- (1,-1.73) -- cycle;
  \draw [fill] (0,0) circle [radius=0.1];
  \draw [fill] (2,0) circle [radius=0.1];
  \draw [fill] (1,-1.73) circle [radius=0.1]; }

\newcommand\triangletwo{
  \draw (0,0) -- (1,1.73) -- (2,0) -- cycle;
  \draw [fill] (0,0) circle [radius=0.1];
  \draw [fill] (2,0) circle [radius=0.1];
  \draw [fill] (1,1.73) circle [radius=0.1]; }

\newcommand\ideal{
  \draw [directed] (0,0) -- (2,0);
  \draw [directed] (0,0) .. controls (0.8,0.8) .. (1,2);
  \draw [directed] (2,0) .. controls (1.2,0.8) .. (1,2);
  \draw (1,2) circle [radius=0.1]; }

\begin{figure}
\begin{tikzpicture}

\begin{scope}[scale=0.5,xshift=-4cm]
	\ideal
	\draw [fill] (0,0) circle [radius=0.1];
	\draw [fill] (2,0) circle [radius=0.1];
	\node at (1,-0.7) {\large $l=1$};
\end{scope}

\begin{scope}[scale=0.5]
	\triangletwo
	\draw [directed] (0,0) .. controls (-0.29,1.09) .. (-1.23,1.87);
	\draw [directed] (1,1.73) .. controls (-.09,1.44) .. (-1.23,1.87);
	\draw (-1.23,1.87) circle [radius=0.1];

	\draw [directed] (2,0) .. controls (2.29,1.09) .. (3.23,1.87);
	\draw [directed] (1,1.73) .. controls (2.09,1.44) .. (3.23,1.87);
	\draw (3.23,1.87) circle [radius=0.1];

	\node at (1,-0.7) {\large $l=2$};
\end{scope}

\begin{scope}[scale=0.5,xshift=6cm]
	\triangletwo
	\begin{scope}[xshift=-1cm,yshift=1.73cm]\triangleone\end{scope}

	\begin{scope}[xshift=-1cm,yshift=1.73cm]\ideal\end{scope}
	\begin{scope}[rotate=120]\ideal\end{scope}
	\begin{scope}[xshift=1cm,yshift=1.73cm,rotate=-60]\ideal\end{scope}

	\node at (0.1,-0.7) {\large ${l=3}$};
\end{scope}

\begin{scope}[scale=0.5,xshift=13.5cm]
	\triangletwo
	\begin{scope}[xshift=-1cm,yshift=1.73cm]\triangleone\end{scope}
	\begin{scope}[xshift=-2cm]\triangletwo\end{scope}

	\begin{scope}[xshift=-1cm,yshift=1.73cm]\ideal\end{scope}
	\begin{scope}[xshift=1cm,yshift=1.73cm,rotate=-60]\ideal\end{scope}
	\begin{scope}[rotate=180]\ideal\end{scope}
	\begin{scope}[xshift=-2cm,rotate=60]\ideal\end{scope}

	\node at (1,0.5) {$E_1$};
	\node at (0,1.2) {$E_2$};
	\node at (-1,0.5) {$E_3$};

	\node at (0.8,-0.8) {\large $l=4$};
\end{scope}

\end{tikzpicture}
\caption{Constructions of $X_l$ for small $l$.  Ideal vertices come from open circles.}
\label{punctured fig}
\end{figure}
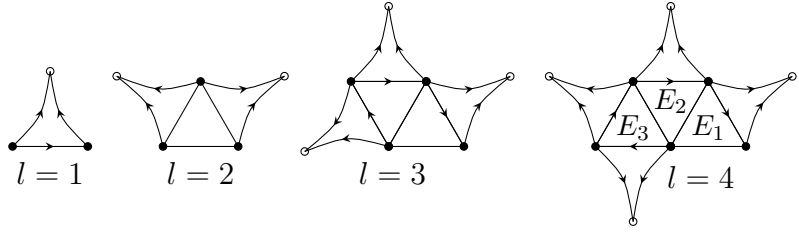

We begin by identifying a subset of the edges of the $E_i$ in pairs so that their union is homeomorphic to a disk, for instance according to the scheme indicated in Figure \ref{punctured fig}, so that (for $l>2$) $E_1$ has exactly two free edges and each $E_i$ has at least one.  An Euler characteristic calculation shows that a total of $2l-4$ edges of the $E_i$ are identified, so the boundary of $\bigcup E_i$ is a union of $l+1$ free edges.  For each free edge $e$ of $\bigcup E_i$ except one belonging to $E_1$, we join one of the $H_i$ to $\bigcup E_i$ by identifying the edge opposite its marked vertex to $e$ via a homeomorphism.  We then finish by identifying the two edges of each $H_i$ that contain its marked vertex to each other by a boundary orientation-reversing homeomorphism.

In the resulting quotient $X_l$, the interior of each $H_i$ forms an open neighborhood of its marked vertex, and its edge opposite the marked vertex descends to a loop based at the non-marked vertex quotient.  There is therefore a unique non-marked vertex quotient, with triangle valence $5l-3$, where $3l-3$ vertices of non-marked triangles are identified together with $2l$ non-marked vertices of marked triangles.

Let $Y_l$ be the complex obtained by joining a non-marked triangle $E_0$ to $X_l$ along its sole free edge (of $E_1$).  Then $Y_l$ is still homeomorphic to a disk with $l$ non-marked vertices in its interior, but its boundary is a union of two edges: the free edges of $E_0$.  The non-marked vertex quotient of $X_l$ has triangle valence $5l-1$ in $Y_l$, with $3l-1$ non-marked triangle vertices identified there, since it picks up an extra two from $E_0$.  The other non-marked vertex quotient in $Y_l$ is the single vertex shared by the two free edges of $E_0$, which therefore has triangle valence one in $Y_l$.\end{example}

\begin{proof}[Proof of Proposition \ref{all triangulations}]  We will begin by re-stating the Proposition separately in the orientable and non-orientable cases.  In the non-orientable case it asserts that for $g\geq 1$ and $n\geq 0$ such that $2-g-n<0$, and any $k\in\mathbb{N}$ that divides both $6(2-g)$ and $n$, there is a closed non-orientable surface of genus $g$ with $n$ marked points which is triangulated with $k+n$ vertices with the following two properties: each marked point is a valence-one vertex; and calling the triangle containing it \textit{marked}, each remaining vertex is a quotient of $2n/k$ marked and
\[ 6 + \frac{6(g-2) + 3n}{k} \] 
non-marked triangle vertices.  Here we recall from the Proposition that the surface is required to have Euler characteristic $\chi+n$, so from the non-orientable case of (\ref{g vs x vs b}) with $b=0$ we have $\chi+n = 2-g$.  Thus $\chi = 2 - g -n < 0$, and $k$ divides both $6(2-g)$ and $n$ if and only if it divides both $6\chi$ and $n$.

The Proposition's orientable case asserts that for $g, n\geq 0$ such that $2-2g-n<0$, and any $k$ dividing both $12(g-1)$ and $n$, that there is a closed, orientable surface of genus $g$ with $n$ marked points which is triangulated in analogous fashion to the non-orientable case except that each non-marked vertex is contained in
\[ 6 + \frac{12(g-1)+3n}{k} \]
non-marked triangle vertices.  The valence computations here and in the non-orientable case are obtained by substituting $2-g-n$ and $2-2g-n$, respectively, for $\chi$ in the formula $6-(6\chi+3n)/k$ in the Proposition's original statement.

From these restatements it is clear that Lemmas \ref{orientable closed} and \ref{nonorientable closed} address the $n=0$ cases, so we assume below that $n>0$.  We first consider the orientable case, beginning with the subcase $k=1$.  For $g\geq 1$ join a copy of the one-holed building block $\Sigma_{g,1}$ of Example \ref{one-hole block}, triangulated with one vertex, to a copy of $X_n$ along their boundaries.  The non-marked vertex is then a quotient of 
  $$(12g-3) + (3n-3) = 12(g-1) + 3n  + 6$$ 
non-marked triangle vertices, and $2n$ marked triangle vertices.  In the case $g=0$ (so with $n\geq 3$), for natural numbers $i$ and $j$ such that $i+j=n$ we join a copy of $X_i$ to a copy of $X_j$ along their boundaries.  The result is homeomorphic to a triangulated sphere with $n$ marked vertices, where the non-marked vertex is a quotient of $(3i-3) + (3j-3) = 3n -6$ non-marked triangle vertices.

Now take $k\geq2$, and suppose $g\geq 2$.  Lemma \ref{orientable closed} supplies a closed, oriented, surface $S_0$ of genus $g$ with no marked points and a $k$-vertex triangulation, where each vertex has valence $12(g-1)/k+6$.  By its construction, $S_0$ is endowed with a collection of disjoint simple closed curves, each a union of one, two, or three edges, that separate it into a union of building blocks.  Orient each such curve on $S_0$, and for an edge $e$ contained in such a curve let $e$ inherit its orientation.  Then cut out the interior of $e$; that is, take the path-completion of $S_0-e$.  The resulting space has a single boundary component which is a union of two edges, and $S_0$ is recovered by identifying them.

Construct $S$ by removing the interior of each such edge $e$, orienting the two new edges to match that of $e$, and joining each new edge to one of a copy of $Y_l$ in orientation-preserving fashion, where $l = n/k$ and the edges of $Y_l$ are oriented pointing away from the free vertex of $E_0$.  Each vertex of $S_0$ is the initial vertex of one oriented edge, and the terminal vertex of one oriented edge.  It therefore picks up one additional non-marked triangle vertex from one copy of $Y_l$ and $3l-1$ from another, as well as $2l$ marked triangle vertices.  This vertex thus has the required valence in $S$.

For the orientable case $g = 1$ and $k \geq 2$ we join one copy of $X_l$ to each boundary component of $\Sigma_{1,k}$, triangulated as in Example \ref{interpolators} with one vertex per boundary component, where $l=n/k$.  Then each non-marked vertex is contained in $9+3l-3 = 3n/k + 6$ non-marked triangle vertices and $2n/k$ vertices from marked triangles.

We finally come to the orientable case $g=0$ and $k\geq 2$.  Since $k$ must divide $12(g-1) = -12$ it can be only $2$, $3$, $4$, $6$ or $12$.  For the case $k=2$, noting that $n$ is thus even, we join one copy of $X_{n/2}$ to each boundary component of the annulus $\Sigma_{0,2}$, triangulated with one vertex per boundary component.  For $k=3$ we produce a sphere $S_0$ by doubling a triangle across its boundary, orient the cycle of triangle edges, and perform the construction described above to yield $S$, using three copies of $Y_{n/3}$.  For $k=6$ and $k=12$ we begin with $\Sigma_{0,2}$ or $\Sigma_{0,4}$, respectively, each triangulated with three vertices per boundary component as in Example \ref{holed spheres}; construct $S_0$ by capping off each boundary component with a single triangle; and proceed similarly.

For the case $k=4$ we first construct a space $Z_l$, $l\in\mathbb{N}$, as follows: join two copies of $Y_l$ along a single edge of $E_0$ in each so that the free vertex of $E_0$ in one is identified to the non-trivial, non-marked vertex quotient in the other.  Then $Z_l$ is homeomorphic to a disk with two vertices on its boundary, each a quotient of $3l$ non-marked triangle vertices and $2l$ non-marked vertices of marked triangles, and $2l$ marked vertices in its interior.  Returning to the $k=4$ subcase, we attach one copy of $Z_{n/4}$ to each boundary component of the annulus $\Sigma_{0,2}$, triangulated with two vertices per boundary component.  This yields a sphere with four non-marked vertices, each a quotient of $3n/4+3$ non-marked triangle vertices as required.  The valence requirements are also easily checked in the other orientable subcases with $g=0$ and $k\geq 2$.

The non-orientable subcase $k=1$ is analogous to the corresponding orientable subcase, with the orientable building block $\Sigma_{g,1}$ replaced by $\Upsilon_{g,1}$ from Example \ref{non-orientable block}.  (There is no analog of the $g=0$ sub-subcase here.)  And the non-orientable subcase $k\geq 2$, $g\geq 3$ is analogous to the orientable subcase $k\geq 2$, $g\geq 2$, with the surface $S_0$ provided here by Lemma \ref{nonorientable closed} instead of \ref{orientable closed}.

The non-orientable subcase $k\geq 2$, $g=2$ is analogous to the orientable subcase $k\geq 2$, $g=1$, but with the $k$-holed torus $\Sigma_{1,k}$ replaced by a $k$-holed Klein bottle (the non-orientable genus-two surface) triangulated with one vertex per boundary component.  Its construction is analogous to that of $\Sigma_{1,k}$ in Example \ref{interpolators}: fill the hole of the triangulated one-holed Klein bottle $\Upsilon_{2,1}$, take a $k$-fold cyclic cover, and remove the preimage of the interior of the added disk.  In all cases so far the valence criteria are straightforward to check.

We finally come to the subcase $k\geq 2$ and $g=1$.  As in the corresponding orientable subcase ($g=0$), we note that possible values of $k$ are tightly restricted by the requirement that $k$ divides $6(g-2) = -6$: it must be either $2$, $3$, or $6$.  For $k=2$ or $3$ we note that the construction of Example \ref{non-orientable block} supplies $\Upsilon_{1,1}$, a one-holed $\mathbb{R}P^2$ triangulated with one, two or three vertices on its boundary.  In each case each vertex is contained in three non-marked triangle vertices.  For $k=2$ we attach a copy of $Z_{n/2}$ to $\Upsilon_{1,1}$, triangulated with two vertices, along their boundaries.  For $k=3$ we triangulate $\Upsilon_{1,1}$ with three vertices, cap off its boundary with a triangle to produce a surface $S_0$, and produce $S$ by suturing in three copies of $Y_{n/3}$ as above.  For $k=6$ we cap off the boundary components of the three-holed $\mathbb{R}P^2$ from Example \ref{holed spheres}, triangulated with two vertices per boundary component, with three copies of $Z_{n/3}$ as defined above.
\end{proof}

\section{Generic non-sharpness}\label{dullity}

In this section we prove Theorem \ref{Vor bound not attained}, that the bound $r_{\chi,n}^k$ of Proposition \ref{Vor bound} is ``generically'' not sharp.  We begin by observing that $r_{\chi,n}^k$ is generically not attained.

\begin{lemma}\label{this ain't it}  For any $\chi<0$ and $k\in\mathbb{N}$ that does not divide $6\chi$, there is no closed hyperbolic surface with Euler characteristic $\chi$ and a packing by $k$ disks of radius $r_{\chi,0}^k$.  For any fixed $\chi< 0$ and $n>0$, there are only finitely many $k\in\mathbb{N}$ for which a complete, finite-area hyperbolic surface with Euler characteristic $\chi$ and $n$ cusps exists which admits a packing by $k$ disks of radius $r_{\chi,n}^k$.\end{lemma}

\begin{proof}  We first consider the closed ($n=0$) case.  The main observation here is that by the equation from Proposition \ref{Vor bound} defining $r_{\chi,n}^k$, $\alpha(r_{\chi,0}^k)$ is an integer submultiple of $2\pi$ if and only if $k$ divides $6\chi$.  By the Proposition, any closed surface that admits a radius-$r_{\chi,0}^k$ disk packing is triangulated by a collection of equilateral triangles that all have vertex angle $\alpha(r_{\chi,n}^k)$.  The total angle around any vertex $v$ of the triangulation is thus $i\cdot \alpha(r_{\chi,n}^k)$, where $i$ is the number of triangle vertices identified at $v$.  But since $v$ is a point of a hyperbolic surface this angle is $2\pi$, whence $k$ must divide $6\chi$.

Let us now fix $\chi<0$ and $n>0$, and recall from Proposition \ref{Vor bound} that $r_{\chi,n}^k>0$ is determined by the equation $f_k(r_{\chi,n}^k) = 2\pi$, where for $\alpha(r) = 2\sin^{-1}\left(\frac{1}{2\cosh r}\right)$ and $\beta(r) = \sin^{-1}\left(\frac{1}{\cosh r}\right)$,
$$ f_k(r) = \left(6-\frac{6\chi+3n}{k}\right)\alpha(r) + \frac{2n}{k}\beta(r) = \left(6 - \frac{6\chi+n}{k}\right)\alpha(r) + \frac{2n}{k}(\beta(r) - \alpha(r)).$$
Rewriting $f_k$ as on the right above makes it clear that for any fixed $r>0$ and $k\in\mathbb{N}$, $f_{k+1}(r) < f_k(r)$.  For since the inverse sine is concave up on $(0,1)$ we have $\beta(r) > \alpha(r)$ for any $r$, and inspecting the equations of (\ref{g vs x vs b}) shows in all cases that $-6\chi-n > 0$.  Moreover, these equations show in all cases that $-6\chi-3n\geq -3$ in all cases, so $f_k$ decreases with $r$ since both $\alpha$ and $\beta$ do.

We now claim that $\alpha(r_{\chi,n}^k)$ strictly increases to $\pi/3$ and $\beta(r_{\chi,n}^k)$ to $\pi/2$ as $k\to\infty$.  For each $k$, by the above we have $f_{k+1}(r_{\chi,n}^k) < f_k(r_{\chi,n}^k) = 2\pi$, so since each $f_k$ decreases with $r$ it follows that $r_{\chi,n}^{k+1} < r_{\chi,n}^k$.  Thus in turn $\alpha(r_{\chi,n}^{k+1})>\alpha(r_{\chi,n}^k)$ and similarly for $\beta$.  Note that for each $r$, $f_k(r) \to 6\alpha(r)$ as $k\to \infty$.  The limit is a decreasing function of $r$ that takes the value $2\pi$ at $r=0$.  Since the sequence $\left\{r_{\chi,n}^k\right\}_{k\in\mathbb{N}}$ is decreasing and bounded below by $0$, it converges to its infimum $\ell$.  If $\ell$ were greater than $0$ then for large enough $k$ we would have $f_k(r_{\chi,n}^k) < f_k(\ell) < 2\pi$, a contradiction.  Thus $r_{\chi,n}^k\to 0$ as $k\to\infty$, and the claim follows by a simple computation.

Now let us recall the geometric meaning of $\alpha$ and $\beta$: $\alpha(r)$ is the angle at any vertex of an equilateral triangle with side length $2r$, and $\beta(r)$ is the angle of a horocyclic triangle with compact side of length $2r$, at either endpoint of this side.  By Proposition \ref{Vor bound}, if there is a complete hyperbolic surface $F$ of finite area with Euler characteristic $\chi$ and $n$ cusps and a packing by $k$ disks of radius $r_{\chi,n}^k$ then $F$ decomposes into equilateral triangles and exactly $n$ horocyclic ideal triangles, all with compact sidelength $2r_{\chi,n}^k$.

For such a surface $F$, since $F$ has $n$ cusps and there are $n$ horocyclic ideal triangles, each horocyclic ideal triangle has its edges identified in $F$ to form a monogon encircling a cusp.  At most $n$ disk centers can be at the vertex of such a monogon, so for $k > n$ there is a disk center which is only at the vertex of equilateral triangles.  The angles around this vertex must sum to $2\pi$, so $\alpha(r_{\chi,n}^k)$ must equal $2\pi/i$ for some $i\in\mathbb{N}$.  But for $k$ large enough we have $\frac{2\pi}{7} < \alpha(r_{\chi,n}^k) < \frac{\pi}{3}$, and this is impossible.
\end{proof}

We will spend the rest of the section showing that among finite-area complete hyperbolic surfaces with any fixed topology, the maximal $k$-disk packing radius does attain a maximum.  To make this more precise requires some notation.  Below for a smooth surface $\Sigma$ of \textit{hyperbolic type} --- that is, which is diffeomorphic to a complete, finite-area hyperbolic surface --- let $\mathfrak{T}(\Sigma)$ be the \textit{Teichm\"uller space} of $\Sigma$.  This is the set of pairs $(F,\phi)$ up to equivalence, where $F$ is a complete, finite-area hyperbolic surface and $\phi\co\Sigma\to F$ is a diffeomorphism (called a \textit{marking}), and $(F',\phi')$ is equivalent to $(F,\phi)$ if there is an isometry $f\co F\to F'$ such that $f\circ\phi$ is homotopic to $\phi'$.  We take the usual topology on $\mathfrak{T}(\Sigma)$, which we will discuss further in Lemma \ref{continuous} below.

\begin{definition}\label{maximy max max} For a surface $\Sigma$ of hyperbolic type and $k\in\mathbb{N}$, define $\maxi_k\co\mathfrak{T}(\Sigma)\to (0,\infty)$ by taking $\maxi_k(F,\phi)$ to be the maximal radius of an equal-radius packing of $F$ by $k$ disks.  Refer by the same name to the induced function on the moduli space $\mathfrak{M}(\Sigma)$ of $\Sigma$.\end{definition}

The \textit{moduli space} $\mathfrak{M}(\Sigma)$ is the quotient of $\mathfrak{T}(\Sigma)$ by the action of the mapping class group, or, equivalently, the collection of complete, finite-area hyperbolic surfaces homeomorphic to $\Sigma$, taken up to isometry, endowed with the quotient topology from $\mathfrak{T}(\Sigma)$.  Since the value of $\maxi_k$ on $(F,\phi)$ depends only on $F$, it induces a well-defined function on $\mathfrak{M}(\Sigma)$.  We are aiming for:

\begin{proposition}\label{there's a max}  For any surface $\Sigma$ of hyperbolic type, and any $k\in\mathbb{N}$, $\maxi_k$ attains a maximum on $\mathfrak{M}(\Sigma)$.\end{proposition}

Assuming the Proposition, we immediately obtain this section's main result:

\begin{theorem}\label{Vor bound not attained}\NotAttained\end{theorem}

This is because for fixed $\chi$ and $n$ there are at most two $n$-punctured surfaces of hyperbolic type with Euler characteristic $\chi$: one orientable and one non-orientable.  By Lemma \ref{this ain't it}, for all but finitely many $k$ the maximum of $\maxi_k$ on each of their moduli spaces (which exists by the Proposition) is less than $r_{\chi,n}^k$.

To prove the Proposition we need some preliminary observations on $\maxi_k$, some basic facts about the topology of moduli space, and some non-trivial machinery due originally to Thurston.

\begin{lemma}\label{maximin} For any surface $\Sigma$ of hyperbolic type, and any $k\in\mathbb{N}$ there exist $c_k>0$ and $\epsilon_k>0$ such that $\maxi_k(F)\geq c_k$ for every $F\in\mathfrak{M}(\Sigma)$, and any packing of $F$ by $k$ disks of radius $\maxi_k(F)$ is contained in the $\epsilon_k$-thick part of $F$.  For each such $F$, there is a packing of $F$ by $k$ disks of radius $\maxi_k(F)$.\end{lemma}

\begin{proof}  There is a universal lower bound of  $\sinh^{-1}(2/\sqrt{3})$ on the maximal injectivity radius of a hyperbolic surface.   (This was shown by Yamada \cite{Yamada}.)  The value of $\maxi_k$ on any hyperbolic surface $F$ is therefore at least the maximum radius $c_k$ of $k$ disks embedded without overlapping in a single disk of radius $\sinh^{-1}(2/\sqrt{3})$.  This proves the lemma's first assertion.

For a maximal-radius packing of $F$ by $k$ disks, the center of each disk is thus contained in the \textit{$c_k$-thick part} $F_{[c_k,\infty)}$ of $F$, the set of $x\in F$ such that the injectivity radius of $F$ at $x$ is at least $c_k$.  The second assertion follows immediately from the (surely standard) fact below:

\begin{claim}  For any $r$ less than the two-dimensional Margulis constant, and any hyperbolic surface $F$, the $r$-neighborhood in $F$ of $F_{[r,\infty)}$ is contained in $F_{[r',\infty)}$, where $r' = \sinh^{-1}\left(\frac{1}{2}(1-e^{-2r})\right)$.\end{claim}

\begin{proof}[Proof of claim] We will use the following hyperbolic trigonometric fact: for a quadrilateral with base of length $\delta$ and two  sides of length $h$, each meeting the base at right-angles, the lengths $\delta$ of the base, $h$ of the sides it meets, and $\ell$ of the remaining side are related by
\begin{align}\label{quad} \sinh(\ell/2) = \cosh h\sinh(\delta/2). \end{align}
Now fix $x\in\partial F_{[r,\infty)}$ and let $U$ be the component of the $\epsilon$-thin part of $F$ that contains $x$, where $\epsilon$ is the two-dimensional Margulis constant.  Suppose first that $U$ is an annulus, and let $h$ be the distance from $x$ to the core geodesic $\gamma$ of $U$.  If $\gamma$ has length $\delta$ then applying (\ref{quad}) in the universal cover we find that there is a closed geodesic arc of length $\ell$ (as defined there) based at $x$ and freely homotopic (as a closed loop) to $\delta$.  It thus follows that $\cosh h = \sinh r/\sinh(\delta/2)$.  Let us now assume that $\delta \leq 2\sinh^{-1}\left(\frac{\sinh r}{\cosh r}\right)$, so that $h \geq r$.  For a point $y$ at distance $h-r$ from $\delta$, the geodesic arc based at $y$ in the free homotopy class of $\delta$ has length $\ell'$ given by
\[ \sinh(\ell'/2) = \cosh(h-r)\sinh(\delta/2) = \sinh r\left(\cosh r - \sqrt{\sinh^2 r - \sinh^2(\delta/2)}\right) \]
(using the ``angle addition'' formula for hyperbolic cosine.)  As a function of $\delta$, $\ell'$ is increasing, and its infimum at $\delta = 0$ satisfies $\sinh(\ell'/2) = \sinh r(\cosh r - \sinh r) = \frac{1}{2}\left(1-e^{-2r}\right)$.  Therefore the injectivity radius at $y$ is at least $r'$ defined in the claim.  Note also that at $\delta = 2\sinh^{-1}\left(\frac{\sinh r}{\cosh r}\right)$ we have $\ell' = \delta$ as expected.  So when $h\leq r$, i.e.~when $U$ is contained in the $r$-neighborhood of $F_{[r,\infty)}$, the injectivity radius at every point of $U$ is at least $\delta/2 > r'$.

Another hyperbolic trigonometric calculation shows that if $U$ is a cusp component of the $\epsilon$-thin part of $F$ then for $y$ at distance $r$ from $x$ along the geodesic ray from $x$ out the cusp, the shortest geodesic arc based at $y$ has length $\ell'$ satisfying the same formula as above.  Hence in this case as well, the injectivity radius is at least $r' = \ell'/2$ at every point of $U$ in the $r$-neighborhood of $F_{[r,\infty)}$.\end{proof}

For the Lemma's final assertion we note that $\maxi_k(F)$ is the supremum of the function on $F^k$ which records the maximal radius of a packing by equal-radius disks centered at $x_1,\hdots,x_k$ for any $(x_1,\hdots,x_k)\in F^k$.  (If $x_i = x_j$ for some $j\neq i$ then we take its value to be $0$.)  This function is clearly continuous on $F^k$.  By the Lemma's first assertion it attains a value of $c_k$, so it approaches its supremum on the compact subset $(F_{[c_k,\infty)})^k$ of $F^k$.\end{proof}

\begin{lemma}\label{continuous} For any surface $\Sigma$ of hyperbolic type, and any $k\in\mathbb{N}$, $\maxi_k$ is continuous on $\mathfrak{T}(\Sigma)$, therefore also on $\mathfrak{M}(\Sigma)$.\end{lemma}

\begin{proof}  We will take the topology on $\mathfrak{T}(\Sigma)$ to be that of approximate isometries as in \cite[3.2.15]{CEG}, with a basis consisting of $K$-neighborhoods of $(F,\phi)\in\mathfrak{T}(\Sigma)$ for $K\in(1,\infty)$.  Such a neighborhood consists of those $(F',\phi')\in\mathfrak{T}(\Sigma)$ for which there exists a $K$-bilipschitz diffeomorphism $\Phi\co F\to F'$ with $\Phi\circ \phi$ homotopic to $\phi'$.  (This is equivalent to the ``algebraic topology'' of \cite[3.1.10]{CEG}, which is in turn equivalent to that induced by the Teichm\"uller metric, cf. \cite[\S 11.8]{FaMa}. In particular Mumford's compactness criterion holds for the induced topology on $\mathfrak{M}(\Sigma)$, see \cite[Prop.~3.2.13]{CEG}.)

It is obvious that $\maxi_k$ is continuous with this topology.  For given a geodesic arc $\gamma$ of length $\ell$ on $F$ and a $K$-bilipschitz diffeomorphism $\Phi\co F\to F'$, the geodesic arc in the based homotopy class of $\Phi\circ \gamma$ has length between $\ell/K$ and $K\ell$.  Now for a packing of $F$ by $k$ disks of radius $\maxi_k(F)$ there are finitely many geodesic arcs of length $\ell\doteq 2\maxi_k(F)$ joining the disk centers, and every other arc based in this set has length at least some $\ell_0>\ell$.  Choosing $K$ near enough to $1$ that $\ell_0/K >K\ell$ we find for every $(F',\phi')$ in the $K$-neighborhood of $(F,\phi)$ that $\maxi_k(F)/K \leq \maxi_k(F') \leq K\maxi_k(F)$.\end{proof} 

The main ingredient in the proof of Proposition \ref{there's a max} is Lemma \ref{the FLoBUS} below.  It was suggested by a referee for \cite{DeB_locmax}, who sketched the proof I give here.

\begin{lemma}\label{the FLoBUS}  Let $\Sigma$ be a surface of hyperbolic type, and fix $r_1 > r_0 >0$, both less than the two-dimensional Margulis constant.  For any essential simple closed curve $\mathfrak{c}$ on $\Sigma$, and any $S\in\mathfrak{T}(\Sigma)$ such that the geodesic representative $\gamma$ of $\frakc$ has length less than $2r_0$ in $S$, there exists $S_0\in\mathfrak{T}(\Sigma)$ where the geodesic representative $\gamma_0$ of $\frakc$ has length between $2r_0$ and $2r_1$, and a path-Lipschitz map (one which does not increase the length of paths) from $S_0-\gamma_0$ onto a region in $S$ containing the complement of the component of the $r_1$-thin part containing $\gamma$.\end{lemma}

\begin{proof}  The construction uses the \textit{strip deformations} described by Thurston \cite{Thurston}.  We will appeal to a follow-up by Papadopolous--Th\'eret \cite{PT} for details and a key geometric assertion.  Below we take $\frakc$ to have geodesic length less than $2r_0$ in $S$. We will assume first that $\frakc$ is non-separating.  The separating case raises a few complications that we will deal with after finishing this case.

To begin the strip deformation construction, cut $S$ open along the geodesic representative $\gamma$ of $\frakc$ to produce a surface with two geodesic boundary components $\gamma_1$ and $\gamma_2$.  Then adjoin a funnel to this surface along each $\gamma_i$ to produce its \textit{Nielsen extension} $\widehat{S}$.  Now fix a properly embedded geodesic arc $\alpha$ in the cut-open surface that meets each of $\gamma_1$ and $\gamma_2$ perpendicularly at an endpoint.  One can obtain such an arc $\alpha$ starting with a simple closed curve $a$ in $S$ that intersects $\gamma$ once: if $a\cap\gamma = \{x\}$ then regarding $a$ as a closed path in $S$ based at $x$, any lift $\tilde{a}$ to the universal cover has its endpoints on distinct components $\tilde\gamma_1$ and $\tilde\gamma_2$ of the preimage of $\gamma$; this lift is then homotopic through arcs with endpoints on $\tilde\gamma_1$ and $\tilde\gamma_2$ to their common perpendicular $\tilde\alpha$, which projects to $\alpha$.

In $\widehat{S}$, $\alpha$ determines a bi-infinite geodesic $\hat{\alpha}$ with one end running out each funnel. We now cut $\widehat{S}$ open along $\hat\alpha$ and glue in an ``$\epsilon$-strip'' $B$ (terminology as in \cite{PT}), which is isometric to the region in $\mathbb{H}^2$ between two ultraparallel geodesics whose common perpendicular (the \textit{core} of $B$) has length $\epsilon$, by an isometry from $\partial B$ to the two copies of $\hat\alpha$ in the cut-open $\widehat{S}$.  This isometry should take the endpoints of the core of $B$ to the midpoints of the copies of $\alpha$ in the two copies of $\hat\alpha$.  Let us call the resulting surface $\widehat{S}_0$.  There is a quotient map $f\co\widehat{S}_0\to\widehat{S}$ taking $B$ to $\hat\alpha$, which is the identity outside $B$ and on $B$ collapses arcs equidistant from the core.  By \cite[Proposition 2.2]{PT}, $f$ is $1$-Lipschitz.  (In \cite{PT}, applying $f$ is called ``peeling an $\epsilon$-strip'', and the respective roles of $\widehat{S}_0$ and $\widehat{S}$ here are played by $\hat{X}$ and $\hat{Y}_B$ there.)

For $i=1,2$, let $\hat\gamma_i$ be the geodesic in $\widehat{S}_0$ that is freely homotopic to $f^{-1}(\gamma_i)$, which is the union of the open geodesic arc $\gamma_i-\alpha$ with an equidistant arc from the core of $B$.  By gluing the endpoints of the core of $B$ to the midpoints of the copies of $\alpha$ we ensured that the $\hat\gamma_i$ have equal length, which depends on $\epsilon$ (we will expand on this assertion below).  We now remove the funnels that the $\hat\gamma_i$ bound in $\widehat{S}_0$, then isometrically identify the resulting boundary components yielding a finite-area surface $S_0$ (the choice of isometry is not important here).

Let $\gamma_0$ be the quotient of the $\hat\gamma_i$ in $S_0$.  The marking $\Sigma\to S$ induces one from $\Sigma$ to $S_0$ that takes $\frakc$ to a simple closed curve with geodesic representative $\gamma_0$.  The restriction of $f$ also induces a map from $S_0-\gamma_0$ to a region in $S$.  This is because the $\hat\gamma_i$ lie inside the preimage $f^{-1}(\overline{S-\gamma})$, where $\overline{S-\gamma}$ refers to the closure $(S-\gamma)\cup\gamma_1\cup\gamma_2$ of $S-\gamma$ inside its Nielsen extension.  This in turn follows from the fact that the nearest-point retraction from $\widehat{S}_0$ to $f^{-1}(\overline{S-\gamma})$, which is modeled on $B\cap \left(\widehat{S}_0-f^{-1}(\overline{S-\gamma})\right)$ by the orthogonal projection to a component of an equidistant locus to the core of $B$ and on the rest of $\widehat{S}_0-f^{-1}(\overline{S-\gamma})$ by the retraction of a funnel to its boundary, is distance-reducing.

Since $f$ does not increase the length of paths, the same holds true for the induced map $S_0-\gamma_0 \to S$.  Let us call $U$ the component of the $r_1$-thin part of $S$ containing $\gamma$.  We will show that $\epsilon$ can be chosen so that the $\hat\gamma_i$, and hence also $\gamma_0$, have length at least $2r_0$ and less than $2r_1$.  Since $f$ is Lipschitz, and the $f(\hat\gamma_i)$ are freely homotopic in $S$ to $\gamma$, it will follow that their images lie in $U$ and thus that the image of $S_0-\gamma_0$ contains $S-U$.

It seems intuitively clear that the length of the $\hat\gamma_i$ increases continuously and without bound with $\epsilon$, and that it limits to the length of $\gamma$ as $\epsilon\to 0$.  If this is so then it is certainly possible to choose $\epsilon$ so that the $\hat\gamma_i$, hence also their quotient in $S_0$, have length strictly between $2r_0$ and $2r_1$.  We will thus spend a few lines describing the hyperbolic trigonometric calculations to justify the assertions above on the length of the $\hat\gamma_i$, completing the proof of the non-separating case.

For $i = 1,2$, $\hat\gamma_i$ is freely homotopic in $\widehat{S}_0$ to the broken geodesic that is the union of the open geodesic arc $\gamma_i - \alpha$ with the geodesic arc contained in the $\epsilon$-strip $B$ that joins its endpoints.  Dropping perpendiculars to $\hat\gamma_i$ from the midpoints of these two arcs divides the region in $\widehat{S}_0$ between $\hat\gamma_i$ and the broken geodesic into the union of two copies of a pentagon $\calp$ with four right angles.  Let $x$ be the length of the base of $\calp$, half the length of $\hat\gamma_i$.  The sides opposite the base have lengths $\delta/2$ and $z$, where $\delta$ and $2z$ are the respective lengths of $\gamma$, and the arc in $B$ joining the endpoints of $\gamma-\alpha$.  See Figure \ref{nonsep}.

\begin{figure}
\begin{tikzpicture}[scale=0.7]


\draw [thick] (0,0) -- (8,0);
\draw [thick] (0,0) -- (0,-0.5);
\draw [thick] (8,0) arc (180:225:2.25);
\draw [thick] (5,0) arc (180:195:7.5);

\draw [thick] (0,-0.5) to [out=0,in=160] (8.66,-1.59);
\draw [thick] (8.66,-1.59) arc (160:170:2);
\draw [thick] (5.256,-1.94) to [out=10,in=170] (8.56,-1.91);

\draw [very thin] (7.9,-1.3) .. controls (7.95,-1.05) .. (8.15,-0.88);
\node at (7.8,-0.9) {\Small$\theta$};

\node at (-0.5,-0.25) {$\epsilon/2$};
\node at (8.5,-0.7) {$z$};
\draw [very thin, gray] (4.7,-1.94) -- (5.1,-1.94);
\draw [very thin, gray] (4.75,-1.35) -- (4.9,-1.84);
\draw [very thin, gray] (4.7,-0.85) -- (4.7,-0.1);
\node at (4.7,-1.1) {$x$};
\node at (10,-1.9) {$\delta/2$};
\draw [<-] (8.75,-1.8) -- (9.5,-1.8);

\node [below] at (3,-0.4) {$h$};
\node at (5.9,-1.3) {\large$\mathcal{P}$};
\node at (4,-0.28) {$\calq$};

\end{tikzpicture}
\caption{The case that $\mathfrak{c}$ is non-separating.}
\label{nonsep}
\end{figure}
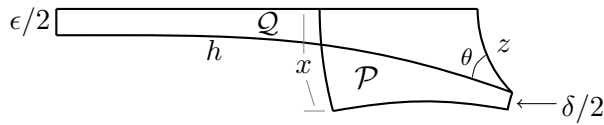

The side of length $z$ is also a side of a quadrilateral $\calq\subset B$ with its other sides in the core of $B$ (with length $\epsilon/2$), the core's perpendicular bisector, and a copy of $\alpha$ (with length $h$).  This quadrilateral has three right angles, and trigonometric formulas for such quadrilaterals (see \cite[VI.3.3]{Fenchel}) describe $z$ and the non-right angle $\theta$ in terms of $\epsilon/2$ and $h$:\begin{align}\label{quadfacts}
	& \sinh z = \cosh h \sinh (\epsilon/2) &
	& \sin\theta  = \cosh(\epsilon/2)/\cosh z \end{align}
The non-right angle of $\calp$ has measure $\pi/2+\theta$, so from trigonometric laws for pentagons with four right angles \cite[VI.3.2]{Fenchel}, we now obtain:
\[ \cosh x  = \cosh h\sinh(\epsilon/2)\sinh(\delta/2) + \cosh(\epsilon/2)\cosh(\delta/2) \]
Recall that $x$ is half the length of the $\hat\gamma_i$.  And indeed we find that $x$ increases with $\epsilon$, continuously and without bound, and it limits to $\delta/2$ (half the length of $\gamma$) as $\epsilon\to 0$.

Now suppose $\frakc$ is separating.  In this case we choose geodesic arcs $\alpha_1$ and $\alpha_2$, each properly embedded in the surface obtained by cutting $S$ along the geodesic representative $\gamma$ of $\frakc$ and meeting the boundary  perpendicularly, one in each component. These may again be obtained from a simple closed curve $a\subset S$, this time one which intersects $\gamma$ twice, such that no isotopy of $a$ reduces the number of intersections. Lifting the arcs $a_1$ and $a_2$ of $a$ on either side of $\gamma$ to the universal cover, separately homotoping each to a common perpendicular between lifts of $\gamma$, then pushing back down to $S$ and cutting along $\gamma$ yields the $\alpha_i$.

After cutting $S$ along $\gamma$, we take $\gamma_i$ to be the geodesic boundary of the component containing $\alpha_i$, for $i=1,2$.  After constructing $\widehat{S}$ as before, we cut it open along both $\hat\alpha_1$ and $\hat\alpha_2$ then, for each $i=1,2$, attach an $\epsilon_i$-strip $B_i$ along the boundary of the component containing $\gamma_i$.  In the resulting disconnected surface $\widehat{S}_0$, for each $i$ let $\hat\gamma_i$ be the geodesic in the free homotopy class of the union of the two arcs of $\gamma_i-\alpha_i$ with the two arcs in $B_i$ joining the endpoints of these arcs that are on the same side of its core.

We now customize $\epsilon_1$ so that $\gamma_1$ has length greater than $2r_0$ but less than $2r_1$, then customize $\epsilon_2$ so that $\gamma_2$ has the same length as $\gamma_1$.  Then as in the previous case we can cut off the funnels that the $\hat\gamma_i$ bound in $\widehat{S}_0$ and glue the resulting boundary components isometrically to produce a surface $S_0$ with a $1$-Lipschitz map to $f^{-1}(K)$, for $K$ as before.  Again what remains is to verify that the length of each $\hat\gamma_i$ increases continuously and without bound with $\epsilon_i$, and that it approaches the length $\delta$ of $\gamma$ as $\epsilon_i\to 0$.

The extra complication in this case arises from the fact alluded to above, that each $\alpha_i$ intersects $\gamma$ twice.  And we cannot guarantee that the points of intersection are evenly spaced along $\gamma$.  So for each $i$ the funnel of $\widehat{S}_0$ bounded by $\hat\gamma_i$, which contains the broken geodesic $f^{-1}(\gamma_i)$, has only one reflective symmetry.  Its fixed locus is the disjoint union of the perpendicular bisectors of the two components of $\gamma_i-\alpha_i$, which divide the region between $\hat\gamma_i$ and $f^{-1}(\gamma_i)$ into two isometric hexagons with four right angles.  One such hexagon $\mathcal{H}$ is pictured in Figure \ref{sep}.

\begin{figure}
\begin{tikzpicture}[scale=0.7]


\draw [thick, gray] (0,0) -- (8,0);
\draw [thick, gray] (0,0.5) -- (0,-0.5);
\draw [thick, gray] (0,0.5) to [out=0,in=200] (8.66,1.59);
\draw [thick, gray] (0,-0.5) to [out=0,in=160] (8.66,-1.59);

\draw [thick] (8,0) arc (180:225:2.25);
\draw [thick] (8,0) arc (180:135:2.25);

\draw [thick] (8.66,1.59) arc (200:188:2.6);
\draw [thick, dashed] (8.66,1.59) arc (200:254:2.6);
\node at (8.8,1.9) {$b$};

\draw [thick] (8.66,-1.59) arc (160:168:2.6);
\draw [thick, dashed] (8.66,-1.59) arc (160:106:2.6);
\draw [<-] (8.75,-1.8) -- (9.25,-1.8);
\node at (9.6,-1.8) {$a$};

\draw [thick] (6,2.2) to [out=-10,in=187] (8.52,2.1);
\draw [thick] (5.75,-1.83) to [out=5,in=170] (8.56,-1.91);
\draw [thick] (6,2.2) to [out=260,in=95] (5.75,-1.83);

\draw [very thin] (7.9,1.3) .. controls (7.95,1.05) .. (8.15,0.88);
\node at (7.8,0.9) {\Small$\theta$};
\draw [very thin] (7.9,-1.3) .. controls (7.95,-1.05) .. (8.15,-0.88);
\node at (7.8,-0.9) {\Small$\theta$};

\node at (-0.5,-0.25) {$\epsilon/2$};
\node [right] at (8,0) {$2z$};
\draw [very thin, gray] (5.25,-1.83) -- (5.65,-1.83);
\draw [very thin, gray] (5.5,2.2) -- (5.9,2.2);
\draw [very thin, gray] (5.45,0.45) -- (5.7,2.1);
\draw [very thin, gray] (5.45,-1.73) -- (5.45,-0.1);
\node at (5.45,0.25) {$x$};

\node [below] at (3,-0.4) {$h$};
\node at (9.5,-0.85) {$y$};
\node at (6.5,1.6) {\large$\mathcal{H}$};

\draw [very thin] (9.87,-0.2) .. controls (9.82,0) .. (9.87,0.2);
\node at (9.6,0) {\small$\psi$};
\node [right] at (10.35,0) {$p$};
\fill (10.35,0) circle [radius=0.07];

\end{tikzpicture}
\caption{The case that $\mathfrak{c}$ is separating and $\epsilon$ is small.}
\label{sep}
\end{figure}
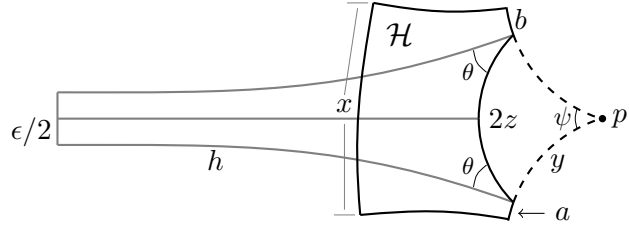

One side of $\mathcal{H}$ has length $x$, half the length of $\hat\gamma_i$ as in the previous case.  The side opposite this one has length $2z$, for $z$ as in (\ref{quadfacts}).  (Here we are suppressing the dependence on $i$ for convenience, but note that $\alpha_1$ and $\alpha_2$ may have different lengths, so ``$z$'' refers to different lengths in the cases $i=1$ and $2$.)  The angle at each endpoint of the side with length $z$ is $\pi/2+\theta$, where again $\theta$ is given by (\ref{quadfacts}) (with the same caveat as for $z$).  The sides meeting this one are contained in $\gamma_i-\alpha_i$, and calling their lengths $a$ and $b$, we have that $a+b = \delta/2$ is half the length of $\gamma$.

When $\epsilon$ is small, the geodesics containing the sides of lengths $a$ and $b$ meet at a point $p$ in $B$.  Their sub-arcs that join $p$ to the endpoints of the side of length $2z$ form an isosceles triangle with two angles of $\pi/2-\theta$.  Applying some hyperbolic trigonometric laws and simplifying gives that the angle $\psi$ at $p$ and the length $y$ of the two equal-length sides respectively satisfy:\begin{align}\label{trifacts}
	& \cos\psi = 2\sinh^2 h\sinh^2(\epsilon/2) - 1 &
	&  \cosh y = \frac{\cosh(\epsilon/2)}{\sqrt{1 - \sinh^2 h \sinh^2(\epsilon/2)}} \end{align}
From these formulas we find in particular that this case holds when $\sinh(\epsilon/2) < 1/\sinh h$.  From the same law for pentagons as in the non-separating case we now obtain:\begin{align*}
	\cosh x & = \sinh(y+a)\sinh(y+b) - \cosh(y+a)\cosh(y+b)\cos\psi \\
	   & = \frac{1}{2}\left[(1-\cos\psi)\cosh(2y+a+b) - (1+\cos\psi)\cosh(a-b) \right] \end{align*}
The second line above is obtained by applying the equation $\sinh x\sinh y = \frac{1}{2}(\cosh(x+y) - \cosh(x-y))$ and its analog for hyperbolic cosine.  Applying hyperbolic trigonometric identities and substituting from (\ref{trifacts}), after some work we may rewrite the right side above as:
\[ \cosh \epsilon\cosh (a+b) + \sinh\epsilon\cosh h\sinh(a+b) + \sinh^2 h\sinh^2(\epsilon/2)(\cosh(a+b) - \cosh(a-b)) \]
This makes it clear that $x$ increases continuously with $\epsilon$, and that $x\to a+b$ as $\epsilon\to 0$.  Recalling that $x$ is half the length of $\hat\gamma_i$ and $a+b = \delta/2$ half that of $\gamma_i$, we have all but one of the desired properties of $\hat\gamma_i$.  But there is a ``phase transition'' at $\epsilon=2\sinh^{-1}(1/\sinh h)$.  Note that as $\epsilon$ approaches this from below, from (\ref{quadfacts}) we have\begin{align*}
	& \sinh z \to \coth h && \sin\theta \to \frac{\cosh h}{\sqrt{\cosh^2 h + \sinh^2 h}} \end{align*}
From Proposition \ref{Vor bound} we now recall that the angle $\beta$ at a finite vertex of a horocyclic ideal triangle and half the length $r$ of its compact side satisfy $\sin\beta = 1/\cosh r$.  One easily verifies that the limit values of $\cosh z$ and of $\sin(\pi/2-\theta) = \cos\theta$ are reciprocal, so geometrically, what is happening as $\epsilon\to2\sinh^{-1}(1/\sinh h)$ is that $p$ is sliding away from the core of $B$ along its perpendicular bisector, reaching an ideal point at $\epsilon = 2\sinh^{-1}(1/\sinh h)$.

For $\epsilon>2\sinh^{-1}(1/\sinh h)$ we therefore expect the geodesics containing $a$ and $b$ to have ultraparallel intersection with $B$.  Their common perpendicular is then one side of a quadrilateral in $B$ with opposite side of length $2z$. If $d$ is the length of this common perpendicular and $y$ the length of the two remaining sides, from standard hyperbolic trigonometric formulas \cite[VI.3.3]{Fenchel} we have:\begin{align*}
	& \cosh d = 2\sinh^2 h\sinh^2(\epsilon/2)-1 &
	& \cosh y = \frac{\sinh z}{\sinh(d/2)} = \frac{\cosh h\sinh(\epsilon/2)}{\sqrt{\sinh^2 h\sinh^2(\epsilon/2)-1}} \end{align*}
We now apply the hyperbolic law of cosines for right-angled hexagons \cite[VI.3.1]{Fenchel}, yielding:\begin{align*}
	\cosh x & = \sinh(y+a)\sinh(y+b) \cosh d-\cosh(y+a)\cosh(y+b) \\
	   & = \frac{1}{2}[\cosh (2y+a+b)(\cosh d - 1) - \cosh(a-b)(\cosh d+1)]\end{align*}
Manipulating this formula along the lines of the previous case now establishes that $x$ increases continuously and without bound with $\epsilon$, and comparing the resulting formula with the previous one shows that its limits coincide as $\epsilon$ approaches $2\sinh^{-1}(1/\sinh h)$ from above and below.  The lemma follows.
\end{proof}

\begin{proof}[Proof of Proposition \ref{there's a max}]  Fix $k\in\mathbb{N}$ and a surface $\Sigma$ of hyperbolic type, and let $r_1 = \min\{\epsilon_0/2,\epsilon_k\}$ and $r_0 = r_1/2$, where $\epsilon_0$ is the two-dimensional Margulis constant and $\epsilon_k$ is as in Lemma \ref{maximin}.  For any $S\in\mathfrak{T}(\Sigma)$ with a geodesic of length less than $2r_0$, we claim that there exists $S'\in\mathfrak{T}(\Sigma)$ with all geodesics of length at least $2r_0$ and $\maxi_k(S')\geq \maxi_k(S)$.

Let $\frakc_1,\hdots,\frakc_n$ be the curves in $\Sigma$ whose geodesic representatives have length less than $2r_0$ in $S$, and let $U$ be the component of the $r_1$-thin part of $S$ containing the geodesic representative $\gamma_1$ of $\frakc_1$.  Lemma \ref{the FLoBUS} supplies some $S_1'\in\mathfrak{T}(\Sigma)$, with a path-Lipschitz map $f$ from $S'_1-\gamma_1'$ to a region in $S$ containing $S-U$, where the geodesic representative $\gamma_1'$ of $\frakc_1$ in $S_1'$ has length between $2r_0$ and $2r_1$.  A packing of $S$ by $k$ disks of radius $\maxi_k(S)$ is entirely contained in $S-U$ by Lemma \ref{maximin} and our choice of $r_1$.  Therefore if $x_1,\hdots,x_k$ are the disk centers in $S$, points $x_1',\hdots,x_k'$ of $S_1'$ mapping to the $x_i$ under $f$ are the centers of a packing of $S_1'$ by disks of the same radius, since $f$ does not increase the length of any based geodesic arc.

We note that for a point $x\in S-U$ where the injectivity radius is at least $r_0$, the injectivity radius is also at least $r_0$ at any $x'\in S_1'$ mapping to $U$: every geodesic arc based at $x'$ that intersects $\gamma_1'$ has length greater than the Margulis constant, and the length of every other arc is decreased by $f$.  It follows that only $\frakc_2,\hdots,\frakc_n$ may have geodesic length less than $2r_0$ in $S_1'$.  Iterating now yields the claim.

Now for a sequence $(S_i)$ of surfaces in $\mathfrak{M}(\Sigma)$ on which $\maxi_k$ approaches its supremum, we use the claim to produce a sequence $(S_i')$, with $\maxi_k(S_i') \geq \maxi_k(S_i)$ for each $i$, such that $S_i'$ has systole length at least $2r_0$ for each $i$.  This sequence has a convergent subsequence in $\mathfrak{M}(\Sigma)$, by Mumford's compactness criterion \cite{Mumford}, and since $\maxi_k$ is continuous on $\mathfrak{M}(\Sigma)$ it attains its maximum at the limit.\end{proof}

\bibliographystyle{plain}
\bibliography{many_disks}

\end{document}